\documentclass[9pt,shortpaper,twoside,web]{ieeecolor}
\usepackage{generic}
\usepackage{cite}
\usepackage{amsmath,amssymb,amsfonts}
\usepackage{algorithmic}
\usepackage{graphicx}
\usepackage{textcomp}

\newtheorem{assumption}{Assumption}
\newtheorem{corollary}{Corollary}
\newtheorem{lemma}{Lemma}
\newtheorem{definition}{Definition}
\newtheorem{remark}{Remark}
\newtheorem{theorem}{Theorem}
\usepackage{subcaption}
\usepackage{booktabs}
\usepackage{tikz}
\usetikzlibrary{shapes,arrows}
\usetikzlibrary{positioning}
\tikzstyle{block} = [draw, fill=white, rectangle, 
    minimum height=3em, minimum width=3em]
\tikzstyle{circ} = [draw, fill=white, circle, minimum size=2.5em]
\tikzstyle{input} = [coordinate]
\tikzstyle{output} = [coordinate]
\tikzstyle{pinstyle} = [pin edge={to-,thin,black}]

\def\BibTeX{{\rm B\kern-.05em{\sc i\kern-.025em b}\kern-.08em
    T\kern-.1667em\lower.7ex\hbox{E}\kern-.125emX}}

\begin{document}
\title{Parameter Estimation in Adaptive Control of Time-Varying Systems Under a Range of Excitation Conditions}
\author{Joseph E. Gaudio, \IEEEmembership{Member, IEEE}, Anuradha M. Annaswamy, \IEEEmembership{Fellow, IEEE},\\Eugene Lavretsky, \IEEEmembership{Fellow, IEEE}, and Michael A. Bolender
\thanks{This work was supported by the Air Force Research Laboratory, Collaborative Research and Development for Innovative Aerospace Leadership (CRDInAL), Thrust 3 - Control Automation and Mechanization grant FA 8650-16-C-2642 and the Boeing Strategic University Initiative; cleared for release, case number 88ABW-2019-4571. \textit{(Corresponding author: Joseph E. Gaudio.)}}
\thanks{J.E. Gaudio and A.M. Annaswamy are with the Department of Mechanical Engineering, Massachusetts Institute of Technology, Cambridge, MA, 02139 USA (email: jegaudio@mit.edu, aanna@mit.edu).}
\thanks{E. Lavretsky is with The Boeing Company, Huntington Beach, CA, 92647 USA (email: eugene.lavretsky@boeing.com).}
\thanks{M.A. Bolender is with the Air Force Research Laboratory, WPAFB, OH, 45433 USA (email: michael.bolender@us.af.mil).}}

\maketitle

\begin{abstract}

This paper presents a new parameter estimation algorithm for the adaptive control of a class of time-varying plants. The main feature of this algorithm is a matrix of time-varying learning rates, which enables parameter estimation error trajectories to tend exponentially fast towards a compact set whenever excitation conditions are satisfied. This algorithm is employed in a large class of problems where unknown parameters are present and are time-varying. It is shown that this algorithm guarantees global boundedness of the state and parameter errors of the system, and avoids an often used filtering approach for constructing key regressor signals. In addition, intervals of time over which these errors tend exponentially fast toward a compact set are provided, both in the presence of finite and persistent excitation. A projection operator is used to ensure the boundedness of the learning rate matrix, as compared to a time-varying forgetting factor. Numerical simulations are provided to complement the theoretical analysis.

\end{abstract}


\section{Introduction}
\label{s:introduction}

Adaptive control is a well established sub-field of control which compensates for parametric uncertainties that occur online so as to lead to regulation and tracking \cite{Narendra2005,Krstic_1995,Ioannou1996,Lavretsky2013}. This is accomplished by constructing estimates of the uncertainties in real-time and ensuring that the closed loop system is well behaved even when these uncertainties are learned imperfectly. Both in the adaptive control and system identification literature, numerous tools for ensuring that the parameter estimates converge to their true values have been derived over the past four decades \cite{Morgan_1977,Morgan_1977a,Anderson_1982,Ljung_1987}. While most of the current literature in these two topics makes an assumption that the unknown parameters are constants, the desired problem statement involves plants where the unknown parameters are varying with time. This paper proposes a new algorithm for such plants.

Parameter convergence in adaptive systems requires a necessary and sufficient condition, denoted as persistent excitation, which ensures that the convergence is uniform in time \cite{Boyd_1983,Boyd_1986,Narendra_1987}. If instead, a weaker condition is enforced where the excitation holds only over a finite interval, then parameter errors decrease only over a finite time. It is therefore of interest to achieve a fast rate of convergence by leveraging any excitation that may be available so that the parameter estimation error remains as small as possible, even in the presence of time-variations. The algorithm proposed in this paper will be shown to lead to such a fast convergence under a range of excitation conditions.

The underlying structure in many of the adaptive identification and control problems consists of a  linear regression relation between two dominant errors in the system \cite{Ioannou1996,Aranovskiy_2016}. Examples include adaptive observers \cite{Lion_1967,Carroll_1973,Luders_1973,Kreisselmeier_1977,Jenkins_2019} and certain classes of adaptive controllers \cite{Narendra2005}. The underlying algebraic relation is often leveraged in order to lead to a fast convergence through the introduction of a time-varying learning rate in the parameter estimation algorithm, which leads to the well-known recursive least squares algorithm \cite{Gauss_1825}. Together with the use of an outer product of the underlying regressor, a matrix of time-varying learning rates is often adjusted to enable fast convergence \cite{Slotine_1989,Li_1990,Slotine_1991,Lion_1967,Kreisselmeier_1977,Goodwin_1987,Krstic_1995a}. In many cases, however, additional dynamics are present in the underlying error model that relates the two dominant errors, which prevents the derivation of the corresponding algorithm and therefore a fast convergence of the parameter estimates. To overcome this roadblock, filtering has been proposed in the literature \cite{Slotine_1989,Li_1990,Slotine_1991,Roy_2018,Cho_2018,Duarte_1989,Duarte_1989a,Lavretsky_2009}. This in turn leads to an algebraic regression, using which corresponding adaptive algorithms are derived in \cite{Slotine_1989,Li_1990,Slotine_1991,Lion_1967,Kreisselmeier_1977,Goodwin_1987,Krstic_1995a} with time-varying learning rates become applicable. In \cite{Roy_2018,Cho_2018}, it is shown that parameter convergence can occur even with finite excitation for a class of adaptive control architectures considered in \cite{Duarte_1989,Duarte_1989a,Lavretsky_2009}. In all of these papers, the underlying unknown parameters are assumed to be constants. The disadvantage of such a filtering approach is that the convergence properties cannot be easily extended to the case when the unknown parameters are time-varying, as the filtering renders the problem intractable. The algorithm that we propose in this paper introduces no filtering of system dynamics, and is directly applied to the original error model with the dynamics intact. As such, we are able to establish conditions for fast decreases of errors, even in the presence of time-varying parameters under varied properties of excitation.

Adaptive control in the presence of time-varying parameters has been studied in \cite{Tsakalis_1987,Tsakalis_1989} using the $\sigma$-modification \cite{Ioannou_1984}, and shown to lead to global boundedness. More recently, this problem has also been studied in \cite{Chowdhary_2013,Chowdhary_2014} using concurrent learning, but with the assumption that the state derivatives from previous time instances are available. While \cite{Parikh_2018} relaxes this assumption, it is at the expense of restricting the parameters to be a constant. In contrast to \cite{Tsakalis_1987,Tsakalis_1989}, our paper can be shown in certain cases to result in convergence to a smaller compact set. In contrast to \cite{Chowdhary_2013,Chowdhary_2014,Parikh_2018}, our paper does not require knowledge of state derivatives and considers time-varying unknown parameters. References \cite{Ge_2003,Chen_2020} further consider adaptive control for plants with a specific structure leveraging Nussbaum gains \cite{Ge_2003} and additional additive terms in the update law \cite{Chen_2020}, while our proposed algorithm does not require such potentially high gain terms.

This paper focuses on the ultimate goal of all adaptive control and identification problems, which is to provide a tractable parameter estimation algorithm for problems where the unknown parameters are time-varying. We will derive such an algorithm that guarantees, in the presence of time-varying parameters, 1) exponentially fast tending of parameter errors and tracking errors to a compact set for a range of excitation conditions, 2) does not require filtering of system dynamics, and 3) is applicable to a large class of adaptive systems. An error model approach as in \cite{Narendra2005,Narendra_1985,Loh_1999} is adopted to describe the underlying class. The algorithm consists of time-varying learning rates in order to guarantee fast parameter convergence. Rather than use a forgetting factor, automatic corrections are provided by the algorithm to ensure that the learning rates are bounded. Additionally, fewer number of integrations are required to implement the algorithm as compared to the existing literature \cite{Roy_2018,Cho_2018,Chowdhary_2013,Chowdhary_2014,Parikh_2018}.

This paper proceeds as follows: Section \ref{s:Preliminaries} presents mathematical preliminaries regarding continuous projection-based operators and definitions of persistent and finite excitation. The underlying problem is introduced in Section \ref{s:Problem_Formulation}. The main algorithm with time-varying learning rates is presented in Section \ref{s:Design}. Stability and convergence properties of this algorithm are established for a range of excitation conditions in Section \ref{s:Stability}. Numerical simulations follow in Section \ref{s:Comparison}. Concluding remarks are provided in Section \ref{s:Conclusion}.

\section{Preliminaries}
\label{s:Preliminaries}

In this paper we use $\lVert\cdot\rVert$ to represent the 2-norm. Definitions, key lemmas, and properties of the Projection Operator (c.f. \cite{Bertsekas_2009,Praly_1991,Pomet_1992,Ioannou1996,Gibson_2012,Lavretsky2013}) are all presented in this section. Proofs of all lemmas can be found in the appendix, and omitted where it is straightforward. Core variables employed throughout are listed in Table \ref{t:Nomenclature} in the appendix.

We begin with a few definitions and properties of convex sets and convex, coercive functions.
\begin{definition}[\hspace{1sp}\cite{Bertsekas_2009}]\label{d:Convex_Set}
    A set $E\subset\mathbb{R}^N$ is convex if $\lambda x+(1-\lambda)y\in E$ for all $x\in E$, $y\in E$, and $0\leq\lambda\leq1$.
\end{definition}
\begin{definition}[\hspace{1sp}\cite{Bertsekas_2009}]\label{d:Convex_Function}
    A function $f:\mathbb{R}^N\rightarrow\mathbb{R}$ is convex if\\$f(\lambda x+(1-\lambda y))\leq\lambda f(x)+(1-\lambda)f(y)$ for all $0\leq\lambda\leq1$.
\end{definition}
\begin{definition}[\hspace{1sp}\cite{Bertsekas_2009}]\label{d:Coercive}
    A function $f(x):\mathbb{R}^N\rightarrow\mathbb{R}$ is said to be coercive if for all sequences $\{x_k\}$, $k\in\mathbb{N}$ with $\lVert x_k\rVert\rightarrow\infty$ then $\lim_{k\rightarrow\infty}f(x_k)=\infty$.
\end{definition}
\begin{lemma}[\hspace{1sp}\cite{Gibson_2012}]\label{l:Xi}
    For a convex function $f(x):\mathbb{R}^N\rightarrow\mathbb{R}$ and any constant $\delta>0$, the subset $\Xi_{\delta}=\{x\in\mathbb{R}^N\mid f(x)\leq\delta\}$ is convex.
\end{lemma}
\begin{lemma}\label{l:bounded_sublevel}
    For a coercive function $f(x):\mathbb{R}^N\rightarrow\mathbb{R}$ and any constant $\delta>0$, any nonempty subset $\Xi_{\delta}=\{x\in\mathbb{R}^N\mid f(x)\leq\delta\}$ is bounded.
\end{lemma}
\begin{corollary}\label{cor:coercive_convex}
    For a coercive, convex function $f(x):\mathbb{R}^N\rightarrow\mathbb{R}$ and a constant $\delta>0$, any nonempty subset $\Xi_{\delta}=\{x\in\mathbb{R}^N\mid f(x)\leq\delta\}$ is convex and bounded.
\end{corollary}
\begin{remark}\label{r:Function_of_Matrix}
    Definitions \ref{d:Convex_Set}, \ref{d:Convex_Function}, \ref{d:Coercive}, Lemmas \ref{l:Xi}, \ref{l:bounded_sublevel}, and Corollary \ref{cor:coercive_convex} hold by simple extension to functions of matrices i.e., $f:\mathbb{R}^{N\times N}\rightarrow\mathbb{R}$.
\end{remark}
\begin{lemma}[\hspace{1sp}\cite{Gibson_2012}]\label{l:theta_star_theta_b}
    For a continuously differentiable convex function $f(x):\mathbb{R}^N\rightarrow\mathbb{R}$ and any constant $\delta>0$, let $\theta^*$ be an interior point of the subset $\Xi_{\delta}=\{x\in\mathbb{R}^N\mid f(x)\leq\delta\}$, i.e. $f(\theta^*)<\delta$, and let a boundary point $\theta$ be such that $f(\theta)=\delta$. Then $(\theta-\theta^*)^T\nabla f(\theta)\geq0$.
\end{lemma}

We now present properties related to projection operators. While some of these properties have been previously reported (c.f. \cite{Praly_1991,Pomet_1992,Ioannou1996,Gibson_2012,Lavretsky2013}), they are included here for the sake of completeness and to help discuss the main result of this paper.
\begin{definition}[\hspace{1sp}\cite{Gibson_2012}]\label{d:Projection_Matrix_Gamma}
    The $\Gamma$-projection operator for general matrices is defined as,
    \begin{equation}\label{e:Projection_Matrix_Gamma}
        \text{Proj}_{\Gamma}(\theta,Y,F)=
        \begin{bmatrix}
        \text{Proj}_{\Gamma}(\theta_1,y_1,f_1)&\hspace{-.15cm}\cdots\hspace{-.15cm}&\text{Proj}_{\Gamma}(\theta_m,y_m,f_m)
        \end{bmatrix}
    \end{equation}
where $\theta=[\theta_1,\ldots,\theta_m]\in\mathbb{R}^{N\times m}$, $Y=[y_1,\ldots,y_m]\in\mathbb{R}^{N\times m}$, $F(\theta)=[f_1,\ldots,f_m]^T:\mathbb{R}^{N\times m}\rightarrow\mathbb{R}^{m}$, $f_j:\mathbb{R}^N\rightarrow\mathbb{R}$ are convex continuously differentiable functions, $0<\Gamma=\Gamma^T\in\mathbb{R}^{N\times N}$ is a symmetric positive definite matrix and $\forall j\in1,\ldots,m$,
    \begin{align}\label{e:Projection_Vector_Gamma}
        \begin{split}
            &\text{Proj}_{\Gamma}(\theta_j,y_j,f_j)=\left\{\begin{array}{ll}
            \hspace{-.15cm}\Gamma y_j-\Gamma\frac{\nabla f_j(\theta_j)(\nabla f_j(\theta_j))^T}{(\nabla f_j(\theta_j))^T\Gamma\nabla f_j(\theta_j)}\Gamma y_jf_j(\theta_j), \\
            & \hspace{-4.75cm}f_j(\theta_j)>0\wedge y_j^T\Gamma\nabla f_j(\theta_j)>0\\
            \hspace{-.15cm}\Gamma y_j, & \hspace{-4.75cm}\text{otherwise}
            \end{array}\right.
        \end{split}
    \end{align}
\end{definition}
\begin{definition}\label{d:Projection_Matrix_PD}
    The projection operator for positive definite matrices is defined as,
    \begin{equation}\label{e:Projection_Matrix_PD}
        \text{Proj}(\Gamma,\mathcal{Y},\mathcal{F})=\left\{\begin{array}{ll}
        \hspace{-.21cm}\mathcal{Y}-\mathcal{F}(\Gamma)\mathcal{Y}, & \hspace{-.28cm}\mathcal{F}(\Gamma)>0\wedge Tr\left[\mathcal{Y}^T\nabla \mathcal{F}(\Gamma)\right]\hspace{-.02cm}>\hspace{-.02cm}0\\
        \hspace{-.21cm}\mathcal{Y}, & \hspace{-.28cm}\text{otherwise}
        \end{array}\right.
    \end{equation}
where $0<\Gamma=\Gamma^T\in\mathbb{R}^{N\times N}$, $\mathcal{Y}\in\mathbb{R}^{N\times N}$ and $\mathcal{F}:\mathbb{R}^{N\times N}\rightarrow\mathbb{R}$ is a convex continuously differentiable function.
\end{definition}
\begin{remark}\label{r:rho}
    The projection operator in \eqref{e:Projection_Matrix_PD} may be expressed more compactly as $\text{Proj}(\Gamma,\mathcal{Y},\mathcal{F})=\rho(t)\mathcal{Y}$ where
    \begin{equation}\label{e:rho}
        \rho(t)=\left\{\begin{array}{ll}
        \hspace{-.15cm}(1-\mathcal{F}(\Gamma)), & \mathcal{F}(\Gamma)>0\wedge Tr\left[\mathcal{Y}^T\nabla \mathcal{F}(\Gamma)\right]>0\\
        \hspace{-.15cm}1, & \text{otherwise}
        \end{array}\right.
    \end{equation}
    From \eqref{e:Projection_Matrix_PD}, \eqref{e:rho}, and as displayed in Figure \ref{f:Projection}, it can be seen that $\rho(t)=1$ on the inside of the projection boundary and $\rho(t)=0$ on the outside edge of the boundary if $Tr\left[\mathcal{Y}^T\nabla \mathcal{F}(\Gamma)\right]>0$.
\end{remark}
\begin{remark}
    An example of a coercive, continuously differentiable convex function commonly used in projection for adaptive control is given by \cite{Gibson_2012}
    \begin{equation}\label{e:Projection_AC}
        f(\theta)=\frac{\lVert \theta\rVert^2-\theta^{*2}_{max}}{2\varepsilon\theta^*_{max}+\varepsilon^2},
    \end{equation}
    where $\theta^*_{max}$ and $\varepsilon$ are positive scalars. It is easy to see that $f(\theta)=0$ when $\lVert \theta\rVert=\theta^*_{max}$ and $f(\theta)=1$ when $\lVert \theta\rVert=\theta^*_{max}+\varepsilon$. This function is commonly used in a projection-based parameter update law to result in a bounded parameter estimate (proven in this paper in Lemma \ref{l:theta_t}). It should be noted that numerous choices other than the one in \eqref{e:Projection_AC} exist for $f$.
\end{remark}
\begin{lemma}\label{l:Trace_theta_tilde_Proj_Y}
    Let $\theta=[\theta_1,\ldots,\theta_m]\in\mathbb{R}^{N\times m}$, $\theta^*=[\theta_1^*,\ldots,\theta_m^*]\in\mathbb{R}^{N\times m}$, $Y=[y_1,\ldots,y_m]\in\mathbb{R}^{N\times m}$, $F(\theta)=[f_1,\ldots,f_m]^T:\mathbb{R}^{N\times m}\rightarrow\mathbb{R}^{m}$, where $f_j:\mathbb{R}^N\rightarrow\mathbb{R}$ are convex continuously differentiable functions, $0<\Gamma=\Gamma^T\in\mathbb{R}^{N\times N}$ is a symmetric positive definite matrix, and $\theta_j^*\in\Xi_{0,j}=\{\theta_j^*\in\mathbb{R}^N\mid f_j(\theta_j^*)\leq0\}$ $\forall j\in1,\ldots,m$, then
    \begin{equation*}
        Tr\left[(\theta-\theta^*)^T\Gamma^{-1}\left(\text{Proj}_{\Gamma}(\theta,Y,F)-\Gamma Y\right)\right]\leq0.
    \end{equation*}
\end{lemma}
\begin{figure}[t]
	\centering
    \begin{tikzpicture}[auto, node distance=1.5cm,>=latex']
    
    \node [input, name=center] {};
    \filldraw[fill=black!15!white, draw=black] (center) ellipse (1.5cm and 1.5cm);
    \filldraw[fill=black!5!white, draw=black] (center) ellipse (1.5*0.75cm and 1.5*0.75cm);
    
    \node [input, name=x1, right=1.25cm of center] {};
    \node [input, name=x1y1, above=0.829156197588850cm of x1] {};
    \node [input, name=x2, right=1.75*1.25cm of center] {};
    \node [input, name=x2y2, above=1.75*0.829156197588850cm of x2] {};
    \draw [->, thick] (x1y1) -- (x2y2);
    \node [input, name=nabla, right=0.4cm of x2y2] {};
    \node[] at (nabla) {$\nabla f(\theta)$};
    
    \node [input, name=x3, right=0.1cm of x1y1] {};
    \node [input, name=x3y3, above=1.5cm of x3] {};
    \draw [->, thick] (x1y1) -- (x3y3);
    \node [input, name=y, right=0.15cm of x3y3] {};
    \node[] at (y) {$y$};
    
    \node [input, name=x4, left=0.75cm of x1y1] {};
    \node [input, name=x4y4, above=0.75*1.25*1.206045378311055cm of x4] {};
    \draw [->, thick] (x1y1) -- (x4y4);
    \node [input, name=Proj, left=0.85cm of x4y4] {};
    \node[] at (Proj) {$\text{Proj}(\theta,y,f)$};
    
    \filldraw[fill=black, draw=black] (x1y1) circle (0.05cm);
    
    \node [input, name=theta_circ, above left=0.45cm of center] {};
    \filldraw[fill=black, draw=black] (theta_circ) circle (0.05cm);
    \node[above left=0.3cm of center] {$\theta^*$};
    
    \node [input, name=t1x, right=2cm of theta_circ] {};
    \node [input, name=t1y, above=1cm of t1x] {};
    \node [input, name=t2y, below=2cm of theta_circ] {};
    \node [input, name=t2x, left=2cm of t2y] {};
    \node [input, name=t3x, right=0.75cm of theta_circ] {};
    \node [input, name=t3y, below=0.5cm of t3x] {};
    \draw (theta_circ) .. controls (t1y) and (t2x) .. (t3y);
    
    \node [input, name=x_Xi, left=3cm of center] {};
    \node [input, name=x_Xi_y1, above=0.5cm of x_Xi] {};
    \node [input, name=x_Xi_y2, below=0.5cm of x_Xi] {};
    \node[] at (x_Xi_y1) {$\{\theta\mid f(\theta)=1\}$};
    \node[] at (x_Xi_y2) {$\{\theta\mid f(\theta)=0\}$};
    \node [input, name=x_Xi_arrow1, left=2cm of center] {};
    \node [input, name=x_Xi_y1_arrow1, above=0.5cm of x_Xi_arrow1] {};
    \node [input, name=x_Xi_y2_arrow1, below=0.5cm of x_Xi_arrow1] {};
    \node [input, name=x_Xi_arrow2, left=1.414213562373095cm of center] {};
    \node [input, name=x_Xi_y1_arrow2, above=0.5cm of x_Xi_arrow2] {};
    \draw [->, very thin] (x_Xi_y1_arrow1) -- (x_Xi_y1_arrow2);
    \node [input, name=x_Xi_arrow3, left=1.007782218537319cm of center] {};
    \node [input, name=x_Xi_y2_arrow2, below=0.5cm of x_Xi_arrow3] {};
    \draw [->, very thin] (x_Xi_y2_arrow1) -- (x_Xi_y2_arrow2);
    
    \node [input, name=y_Xi1, below=0.85cm of center] {};
    \node[] at (y_Xi1) {$\Xi_{0}$};
    
    \end{tikzpicture}
	\caption{(adapted from \cite{Gibson_2012}) $\Gamma$-Projection operator in $\mathbb{R}^2$ with $\Gamma=I$. Uncertain parameter $\theta^*(t)\in\Xi_{0}=\{\theta^*\in\mathbb{R}^2\mid f(\theta^*)\leq0\}$.}
	\label{f:Projection}
\end{figure}

The following lemma lists a key property related to matrix inversion in the presence of time-variations.
\begin{lemma}\label{l:Gamma_dot}
    For a matrix $0<\Gamma(t)=\Gamma^T(t)\in\mathbb{R}^{N\times N}$, the following holds: $\frac{d}{dt}\left(\Gamma^{-1}(t)\right)=-\Gamma^{-1}(t)\dot{\Gamma}(t)\Gamma^{-1}(t)$.
\end{lemma}

A central component of this paper is with regards to excitation of a regressor for which two definitions are provided.
\begin{definition}[\hspace{1sp}\cite{Narendra2005}]\label{d:PE}
    A bounded function $\phi:[t_0,\infty)\rightarrow\mathbb{R}^N$ is persistently exciting (PE) if there exists $T\hspace{-.045cm}>\hspace{-.045cm}0$ and $\alpha\hspace{-.045cm}>\hspace{-.045cm}0$ such that
    \begin{equation*}
        \int_t^{t+T}\phi(\tau)\phi^T(\tau)d\tau\geq\alpha I,\quad \forall t\geq t_0.
    \end{equation*}
\end{definition}
\begin{definition}[adapted from \cite{Roy_2018,Cho_2018}]\label{d:FE}
    A bounded function $\phi:[t_0,\infty)\rightarrow\mathbb{R}^N$ is finitely exciting (FE) on an interval $[t_1,~t_1+T]$ if there exists $t_1\geq t_0$, $T>0$, and $\alpha>0$ such that
    \begin{equation*}
        \int_{t_1}^{t_1+T}\phi(\tau)\phi^T(\tau)d\tau\geq\alpha I.
    \end{equation*}
\end{definition}
In both Definitions \ref{d:PE} and \ref{d:FE}, the degree of excitation is given by $\alpha$. It can be noted that the PE condition in Definition \ref{d:PE} pertains to a property over a moving window for all $t\geq t_0$, whereas the FE condition in Definition \ref{d:FE} pertains to a single interval $[t_1,~t_1+T]$.

\section{Adaptive Control of a Class of Plants with Time-Varying Parameters}
\label{s:Problem_Formulation}

\begin{table}[t]
\caption{Adaptive Control Systems with a Common Structure}
\centering
\begin{tabular}{lll}
\toprule
Name & Error Model & $Y$ \\
\midrule
State Feedback MRAC \cite{Narendra2005} & \hspace{.16cm}$\dot{e}=Ae+B\tilde{\theta}^T\phi$ & $-\phi e^TPB$ \\
  & $e_y=e$ & \vspace{.2cm}\\
Output Feedback MRAC & \hspace{.16cm}$\dot{e}=Ae+B\tilde{\theta}^T\phi$ & $-\phi e_y$ \\
$W(s)$ A.S, SPR \cite{Narendra2005} & $e_y=Ce$ & \vspace{.2cm}\\
Output Feedback MRAC & \hspace{.17cm}$\epsilon=\tilde{\theta}^T\zeta$ & $-\zeta\epsilon$ \\
$W(s)$ A.S., not SPR \cite{Narendra2005} &  & \vspace{.2cm}\\
Nonlinear Adaptive & \hspace{.16cm}$\dot{e}=A_z(e,\theta,t)e$ & $-\phi e$ \\
Backstepping \cite{Krstic_1995} & \hspace{.16cm}$~\quad+\tilde{\theta}^T\phi(e,\theta,t)$ & \\
Relative Degree $\leq 2$ & $e_y=e$ & \vspace{.2cm}\\
\bottomrule
\end{tabular}
\label{t:Common_Structure}
\end{table}

Large classes of problems in adaptive identification and control can be represented in the form of differential equations containing two errors, $e(t)\in\mathbb{R}^n$ and $\tilde{\theta}(t)\in\mathbb{R}^{N\times m}$. The first is an error that represents an identification error or tracking error. The second is the underlying parameter error, either in estimation of the plant parameter or the control parameter. The parameter error is commonly expressed as the difference between a parameter estimate $\theta$ and the true unknown value $\theta^*$ as $\tilde{\theta}(t)=\theta(t)-\theta^*(t)$. The differential equations which govern the evolution of $e(t)$ with $\tilde{\theta}(t)$ are referred to as error models \cite{Narendra2005,Narendra_1985,Loh_1999}, and provide insight into how stable adaptive laws for adjusting the parameter error can be designed for a large class of adaptive systems. The class of error models we focus on in this paper is of the form
\begin{align}\label{e:error_model_general}
    \begin{split}
        \dot{e}(t)&=g_1(e(t),\phi(t),\theta(t),\theta^*(t)),\\
        e_y(t)&=g_2(e(t),\phi(t),\theta(t),\theta^*(t)),
    \end{split}
\end{align}
where the regressor $\phi(t)\in\mathbb{R}^N$ and $e_y(t)\in\mathbb{R}^p$ is a measurable error at each $t$. The corresponding adaptive law for adjusting $\theta$ is assumed to be of the form
\begin{equation}\label{e:theta_dot_nom}
    \dot{\theta}(t)=\Gamma_0 Y(e_y(t),\phi(t),\theta(t)),
\end{equation}
where $Y$ is a known function that is implementable at each $t$ and $\Gamma_0\in\mathbb{R}^{N\times N}$ is a symmetric positive definite matrix referred to as the learning rate. In addition, for a given $g_1$ and $g_2$, $Y$ is chosen so that $e(t)=0$, $\tilde{\theta}(t)=0$ is an equilibrium point of the system, when $\theta^*$ is a constant. We consider all classes of adaptive systems that can be expressed in the form of \eqref{e:error_model_general} and \eqref{e:theta_dot_nom} where $g_1$, $g_2$, $Y$, and $\Gamma_0$ are such that all solutions are bounded for constant $\theta^*$, and where $\lim_{t\rightarrow\infty}e(t)=0$. In particular, we assume that $g_1$, $g_2$, and $Y$ are such that a quadratic Lyapunov function candidate
\begin{equation}\label{e:Lypaunov_Gamma_0_V}
    V(t)=e^T(t)Pe(t)+Tr\left[\tilde{\theta}^T(t)\Gamma_0^{-1}\tilde{\theta}(t)\right],
\end{equation}
yields a derivative for the case of constant $\theta^*$ as
\begin{equation}\label{e:Lypaunov_Gamma_0_V_dot}
    \dot{V}(t)\leq-e^T(t)Qe(t)-2Tr\left[\tilde{\theta}^T(t)Y(t)\right]+2Tr\left[\tilde{\theta}^T(t)\Gamma_0^{-1}\dot{\theta}(t)\right],
\end{equation}
where $P$ and $Q$ are symmetric positive definite matrices. Due to the choice of the adaptive law in \eqref{e:theta_dot_nom}, it follows therefore $\dot{V}(t)\leq-e^T(t)Qe(t)$. Further conditions on $g_1$, $g_2$, and $Y$ guarantee that $e(t)\rightarrow0$ as $t\rightarrow\infty$. We formalize this assumption below:
\begin{assumption}[Class of adaptive systems]\label{a:e_e_y_V_V_dot}
    For the case of a constant unknown parameter ($\dot{\theta}^*(t)=0$), the error model in \eqref{e:error_model_general} and the adaptive law in \eqref{e:theta_dot_nom} are such that they admit a Lyapunov function $V$ as in \eqref{e:Lypaunov_Gamma_0_V} which satisfies the inequality in \eqref{e:Lypaunov_Gamma_0_V_dot}.
\end{assumption}

Several adaptive systems that have been discussed in the literature satisfy Assumption \ref{a:e_e_y_V_V_dot}, some examples of which are shown in Table \ref{t:Common_Structure}. They include plants where state feedback is possible and certain matching conditions are satisfied, and where only outputs are accessible and a strictly positive real transfer function $W(s)$ can be shown to exist. For a SISO plant that is minimum phase, the same assumption can be shown to hold as well. Finally, for a class of nonlinear plants, where the underlying relative degree does not exceed two, Assumption \ref{a:e_e_y_V_V_dot} once again can be shown to be satisfied.

\subsection{Problem Formulation}
\label{ss:AC_TV_Parameters}

The class of error models we consider is of the form \eqref{e:error_model_general}, where $\theta^*(t)$, the time-varying unknown parameter, is such that if $\theta(t)\equiv\theta^*(t)$, then the solutions of \eqref{e:error_model_general} are globally bounded. This is formalized in the following assumption:
\begin{assumption}[Uncertainty variation]\label{a:theta_star}
    The uncertainty, $\theta^*(t)$, in \eqref{e:error_model_general} is such that $||\theta^*(t)||\leq\theta^*_{max}$, $\forall t\geq t_0$. In addition, its time derivative, $\dot{\theta}^*(t)$, is assumed to be bounded, i.e. $\lVert\dot{\theta}^*(t)\rVert\leq \theta^*_{d,max}$, $\forall t\geq t_0$. Furthermore, if $\theta(t)=\theta^*(t)$, $\forall t\geq t_0$ then the solutions of \eqref{e:error_model_general} are globally bounded.
\end{assumption}

The problem that we address in this paper is the determination of an adaptive law similar to \eqref{e:theta_dot_nom} for all error models of the form \eqref{e:error_model_general} where Assumptions \ref{a:e_e_y_V_V_dot} and \ref{a:theta_star} hold. Our goal is to ensure global boundedness of solutions of \eqref{e:error_model_general} and exponentially fast tending of both $e(t)$ and $\tilde{\theta}(t)$ to a compact set with finite excitation.

\section{Update Law with a Time-Varying Learning Rate}
\label{s:Design}

The adaptive law that we propose is a modification of \eqref{e:theta_dot_nom} with a \emph{time-varying learning rate} $\Gamma(t)$ as $\dot{\theta}(t)=\Gamma(t) Y(e_y(t),\phi(t),\theta(t))$. To ensure a bounded $\theta(t)$, we include a projection operator in this adaptive law which is stated compactly as
\begin{equation}\label{e:theta_dot_AR}
    \dot{\theta}(t)=\text{Proj}_{\Gamma(t)}\left(\theta(t),Y(t),F\right),\quad \theta(t_0)\in\Xi_{1}.
\end{equation}
where $\text{Proj}_{\Gamma(t)}(\cdot,\cdot,\cdot)$ is defined as in Definition \ref{d:Projection_Matrix_Gamma}. The $\Gamma$-projection operator in \eqref{e:theta_dot_AR} uses $F(\theta)=[f_1,\ldots,f_m]^T:\mathbb{R}^{N\times m}\rightarrow\mathbb{R}^{m}$, where $f_j(\theta_j):\mathbb{R}^N\rightarrow\mathbb{R}$ are coercive, continuously differentiable convex functions. Define the subsets $\Xi_{\delta,j}=\{\theta_j\in\mathbb{R}^N\mid f_j(\theta_j)\leq\delta\}$, $\forall j\in1,\ldots,m$, and $\Xi_{\delta}=\{M\in\mathbb{R}^{N\times m}\mid M_j\in\Xi_{\delta,j},~\forall j\in1,\ldots,m\}$. Via Assumption \ref{a:theta_star}, each $f_j$ are chosen such that $||\theta^*(t)||\leq\theta^*_{max}$ and $\delta=0$ corresponds to $\theta_j^*(t)\in\Xi_{0,j}$, $\forall j\in1,\ldots,m$, $\forall t\geq t_0$.

The time-varying learning rate $\Gamma(t)$ is adjusted using the projection operator for positive definite matrices (see Definition \ref{d:Projection_Matrix_PD}) as
\begin{align}\label{e:Gamma_Update}
    \begin{split}
    \dot{\Gamma}(t)&=\lambda_{\Gamma}\text{Proj}\left(\Gamma(t),\mathcal{Y}(t),\mathcal{F}\right),\quad \Gamma(t_0)=\Gamma_{t_0},\\
    \mathcal{Y}(t)&=\Gamma(t)-\kappa\Gamma(t)\Omega(t)\Gamma(t),
    \end{split}
\end{align}
where $\lambda_{\Gamma}$, $\kappa$ are positive scalars and $\Omega(t)\in\mathbb{R}^{N\times N}$. $\Gamma_{t_0}$ is a symmetric positive definite constant matrix chosen so that $\Gamma_{t_0}\in\Upsilon_1=\{\Gamma\in\mathbb{R}^{N\times N}\mid \mathcal{F}(\Gamma)\leq1\}$, where $\mathcal{F}(\Gamma):\mathbb{R}^{N\times N}\rightarrow\mathbb{R}$ is a coercive, continuously differentiable convex function. Lemma \ref{l:bounded_sublevel} implies there exists a constant $\Gamma_{max}>0$ such that $\lVert\Gamma\rVert\leq\Gamma_{max}$ for all $\Gamma\in\Upsilon_1$. We assume that $\mathcal{F}$ is chosen so that $\mathcal{F}(\Gamma)=1$ for all $\lVert\Gamma\rVert=\Gamma_{max}$. It should be noted that a large $\Omega(t)$ contributes to a decrease in $\Gamma(t).$

Finally the matrix $\Omega(t)$ is adjusted as
\begin{equation}\label{e:Omega_Update}
    \dot{\Omega}(t)=-\lambda_{\Omega}\Omega(t)+\lambda_{\Omega}\frac{\phi(t)\phi^T(t)}{1+\phi^T(t)\phi(t)},\quad \Omega(t_0)=\Omega_0,
\end{equation}
where $\Omega_0$ is a symmetric positive semi-definite matrix with $0\leq\Omega_0\leq I$ and denotes a filtered normalized regressor matrix. $\lambda_{\Gamma}$ and $\lambda_{\Omega}$ are arbitrary positive scalars and $\kappa$ is chosen so that $\kappa>\Gamma_{max}^{-1}$. These scalars represent various weights of the proposed algorithm. The main contribution of this paper is the adaptive law in \eqref{e:theta_dot_AR}, \eqref{e:Gamma_Update}, and \eqref{e:Omega_Update}, which will be shown to result in bounded solutions in Section \ref{s:Stability} which tend exponentially fast to a compact set if $\phi(t)$ is finitely exciting. If in addition, $\phi(t)$ is persistently exciting, exponentially fast convergence to a compact set will occur $\forall t\geq t_0$.

\begin{remark}
    It should be noted that while different aspects of the algorithm in \eqref{e:theta_dot_AR}, \eqref{e:Gamma_Update}, and \eqref{e:Omega_Update} have been explored in the literature, a combined algorithm as presented and analyzed here has not been reported thus far. For example, filtered regressor outer products are considered in \cite{Krstic_1995a,Roy_2018,Cho_2018}, but parameters are assumed to be constant. It is the fact that we have the use of $\Gamma$ in \eqref{e:Gamma_Update} in a specific manner (see choice of $\mathcal{Y}(t)$), the fact that we are using projections to contain $\Gamma$ within a bounded set, and that we are using a filtered version of $\phi\phi^T$ together with normalization to adjust $\Omega$ as in \eqref{e:Omega_Update} that enables our proposed algorithm to have desirable convergence properties, over a range of excitation conditions.
\end{remark}
\begin{remark}
    The projection operator employed in \eqref{e:Gamma_Update} is one method to bound the time-varying learning rate. Instead of \eqref{e:Gamma_Update}, one can also use a time-varying forgetting factor, $(1-\lVert\Gamma(t)\rVert/\Gamma_{max})$, to provide for $\lVert\Gamma(t)\rVert\leq\Gamma_{max}$ in an update of the form $\dot{\Gamma}(t)=\lambda_{\Gamma}(1-\lVert\Gamma(t)\rVert/\Gamma_{max})\left[\Gamma(t)-\kappa\Gamma(t)\Omega(t)\Gamma(t)\right]$. While the time-varying forgetting factor also achieves a bounded $\lVert\Gamma(t)\rVert$, it is more conservative than the projection operator in \eqref{e:Gamma_Update} as it is always active. In comparison, the projection operator as in \eqref{e:Gamma_Update} only provides limiting action if $\Gamma(t)$ is in a specified boundary region and the direction of evolution of $\Gamma(t)$ causes $\mathcal{F}(\Gamma)$ to increase. An equivalence between time-varying forgetting factors and projection operators may be drawn using the square root of the function in \eqref{e:Projection_AC} with $\theta^*_{max}=0$, $\varepsilon=\Gamma_{max}$, and the limiting action always active.
\end{remark}
\begin{remark}
    It can be noted that while the regressor normalization in \eqref{e:Omega_Update} is optional for linear regression systems \cite{Slotine_1989,Li_1990}, it is required for general adaptive control problems in the presence of system dynamics as the regressor cannot be assumed to be bounded.
\end{remark}
\begin{remark}
    The standard parameter update in \eqref{e:theta_dot_nom} requires $N\times m$ integrations to adjust the $N\times m$ parameters $\theta$. Given that the updates for both $\Gamma$ and $\Omega$ result in symmetric matrices, an additional $N(N+1)/2$ integrations are required for each update for a total increase of $N(N+1)$ integrations.
\end{remark}

\begin{table}[t]
\caption{Norm of Signals in Phases of Excitation Propagation}
\centering
\begin{tabular}{cccccc}
\toprule
$t\in$ & $[t_1,t_2]$ & $t_2$ & $[t_2,t_3]$ & $[t_3,t_4]$ \\
\midrule
$\lVert\phi(t)\phi^T(t)\rVert$ & $\geq0$ &  &  &  \\
$\lVert\int_{t_1}^{t_2}\phi(\tau)\phi^T(\tau)d\tau\rVert$ &  & $\geq\alpha$ &  &  \\
$\lVert\Omega(t)\rVert$ &  &  & $\geq\Omega_{FE}$ &  \\
$\lVert\Gamma(t)\rVert$ &  &  &  & $\leq\Gamma_{FE}$ \\
$\rho(t)$ &  &  &  & $\geq\rho_0$ \\
\bottomrule
\end{tabular}
\label{t:Comparison_Alg}
\end{table}

\section{Stability and Convergence Analysis}
\label{s:Stability}

We now state and prove the main result. The following assumption is needed for discussion of a finite excitation. We define an excitation level $\alpha_0$ on an interval $[t_1,t_2]$ as
\begin{equation}\label{e:minimum_FE}
    \alpha_0=\frac{k_{\Omega}d}{\kappa\Gamma_{max}\rho_{\Omega}\lambda_{\Omega}\exp(-\lambda_{\Omega}(t_2-t_1))},
\end{equation}
where $\rho_{\Omega}\in(0,1)$, $k_{\Omega}>1$, and $d=\max_{\tau\in[t_1,t_2]}\{1+\lVert\phi(\tau)\rVert^2\}$.
\begin{assumption}[Finite excitation]\label{a:FE}
    There exists a time $t_1\geq t_0$ and a time $t_2>t_1$ such that the regressor $\phi(t)$ in \eqref{e:error_model_general} is finitely exciting over $[t_1,t_2]$, with excitation level $\alpha\geq \alpha_0$.
\end{assumption}

\subsection{Propagation of Excitation and Boundedness of Information Matrix, Time-Varying Learning Rate}
\label{ss:Gamma_Omega_Propagation}

We first prove a few important properties of $\Omega(t)$ and $\Gamma(t)$ under different excitation conditions.

\begin{lemma}\label{l:Omega}
    For the algorithm in \eqref{e:Omega_Update}, it follows that for any $\phi(t)$,
    \begin{enumerate}
        \item $\Omega(t)\geq0$, $\forall t\geq t_0$,
        \item $\Omega(t)\leq I$, $\forall t\geq t_0$.
    \end{enumerate}
    If in addition $\phi$ is finitely exciting as in Assumption \ref{a:FE}, then
    \begin{enumerate}\setcounter{enumi}{2}
        \item $\Omega(t)\geq\Omega_{FE}I>(1/(\kappa\Gamma_{max}))I$, $\forall t\in[t_2,~t_3]$,
    \end{enumerate}
    where $\Omega_{FE}=(k_{\Omega}/(\kappa\Gamma_{max}))$ and $t_3=t_2-(\ln{\rho_{\Omega}})/\lambda_{\Omega}$. If in addition $\phi$ is persistently exciting $\forall t\geq t_1\geq t_0$ (see Definition \ref{d:PE}), with interval $T$ and level $\alpha\geq\alpha_0'$, $t_2'=t_1+T$, then
    \begin{enumerate}\setcounter{enumi}{3}
        \item $\Omega(t)\geq\Omega_{FE}I>(1/(\kappa\Gamma_{max}))I$, $\forall t\geq t_2'$,
    \end{enumerate}
    $\alpha_0'=\alpha_0\exp(-\lambda_{\Omega}(t_2-t_2'))d'/d$, $d'=\max_{\tau\geq t_1}\{1+\lVert\phi(\tau)\rVert^2\}$.
\end{lemma}

\begin{lemma}\label{l:Gamma}
    The solutions of \eqref{e:Gamma_Update} and \eqref{e:rho} satisfy the following:
    \begin{enumerate}
        \item $\Gamma(t)\leq\Gamma_{max}I$, $\Gamma^{-1}(t)\geq\Gamma_{max}^{-1}I>0$, $\forall t\geq t_0$,
        \item $\rho(t)\in[0,1]$, $\forall t\geq t_0$,
        \item $\Gamma(t)\geq\Gamma_{min}I>0$, $\Gamma^{-1}(t)\leq\Gamma_{min}^{-1}I$, $\forall t\geq t_0$,
    \end{enumerate}
    where $\Gamma_{min}=1/(\max eig(\Gamma^{-1}_{t_0})+\kappa)$. If in addition $\phi$ is finitely exciting as in Assumption \ref{a:FE}, then there exists a $\rho_0\in(0,1]$ such that
    \begin{enumerate}\setcounter{enumi}{3}
        \item $\Gamma(t)\leq\Gamma_{FE}I<\Gamma_{max}I$, $\Gamma^{-1}(t)\geq\Gamma_{FE}^{-1}I>0$, $\forall t\in[t_3,~t_4]$,
        \item $\rho(t)\geq\rho_0>0$, $\forall t\in[t_3,~t_4]$,
    \end{enumerate}
    where $\Gamma_{FE}=\rho_{\Gamma}^{-1}\Gamma_{t_3}$, $\rho_{\Gamma}\in((\Gamma_{t_3}/\Gamma_{max}),1)$, $\Gamma_{t_3}=\Gamma(t_3)<\Gamma_{max}$, and $t_4=t_3-(\ln{\rho_{\Gamma}})/\lambda_{\Gamma}$. If in addition $\phi$ is persistently exciting $\forall t\geq t_1\geq t_0$ (see Definition \ref{d:PE}), with interval $T$ and level $\alpha\geq\alpha_0'$, then there exists a $\rho_0\in(0,1]$, $t_3'>t_2'$, and $\Gamma(t_3')\leq\Gamma_{PE}<\Gamma_{max}$ such that
    \begin{enumerate}\setcounter{enumi}{5}
        \item $\Gamma(t)\leq\Gamma_{PE}I<\Gamma_{max}I$, $\Gamma^{-1}(t)\geq\Gamma_{PE}^{-1}I>0$, $\forall t\geq t_3'$,
        \item $\rho(t)\geq\rho_0>0$, $\forall t\geq t_3'$.
    \end{enumerate}
\end{lemma}

The properties of $\Omega$ and $\Gamma$ for a persistently exciting $\phi$ are relatively well known. For a finitely exciting $\phi$, it should be noted that after a certain time elapses, the lower bound for $\Omega$ is realized. This propagation is illustrated in Table \ref{t:Comparison_Alg}.

The choice of the finite excitation level $\alpha_0$ in Assumption \ref{a:FE} enables a fast convergence rate as follows: The denominator $\kappa\Gamma_{max}$ in $\alpha_0$ ensures that the update in \eqref{e:Gamma_Update} pushes $\Gamma(t)$ away from $\Gamma_{max}$, $\rho_{\Omega}$ provides for a bound for $\Omega$ away from a minimum value, and $\lambda_{\Omega}\exp(-\lambda_{\Omega}(t_2-t_1))$ accounts for excitation propagation through \eqref{e:Omega_Update}. The numerator scaling $d$ accounts for the normalization in \eqref{e:Omega_Update}, and $k_{\Omega}$ provides for a bound away from a minimum excitation level.

\subsection{Stability and Convergence Analysis}
\label{ss:Lyapunov}

With the properties of the learning rate and filtered regressor above, we now proceed to the main theorem. The following lemma and corollary state important properties of the parameter estimate $\theta$.
\begin{lemma}\label{l:theta_t}
    The update for $\theta(t)$ in \eqref{e:theta_dot_AR} guarantees that there exists a $\theta_{max}$ such that $\lVert\theta(t)\rVert\leq\theta_{max}$, $\forall t\geq t_0$.
\end{lemma}
\begin{corollary}\label{cor:tilde_theta_max}
    Under Assumption \ref{a:theta_star}, the update for $\theta(t)$ in \eqref{e:theta_dot_AR} provides for a constant $\tilde{\theta}_{max}$ such that $\lVert\tilde{\theta}(t)\rVert\leq\tilde{\theta}_{max}$, $\forall t\geq t_0$.
\end{corollary}

The following definitions are useful for stating the main result in Theorem \ref{theorem:Convergence}. Define scalars $\upsilon(t)$ and $\eta(t)$ as
\begin{equation}\label{e:upsilon}
    \upsilon(t)=\lambda_{\Gamma}\rho(t)\kappa \lVert\Omega(t)\rVert\lVert\tilde{\theta}(t)\rVert^2+2\Gamma_{min}^{-1}\lVert\tilde{\theta}(t)\rVert\lVert\dot{\theta}^*(t)\rVert,
\end{equation}
\begin{equation}\label{e:eta_t}
    \eta(t)=\min\left\{q_0,\lambda_{\Gamma}\rho(t)\Gamma_{max}^{-1}\right\}/\max\left\{p_{max},\Gamma_{min}^{-1}\right\}.
\end{equation}
It is easy to see that $0\leq \eta(t)$ and $0\leq\upsilon(t)\leq\upsilon_{max}$, where
\begin{equation}\label{e:upsilon_max}
    \upsilon_{max}=\lambda_{\Gamma}\kappa\tilde{\theta}_{max}^2+2\Gamma_{min}^{-1}\tilde{\theta}_{max}\theta^*_{d,max}.
\end{equation}
Define $\eta_0$ as
\begin{equation}\label{e:eta_0}
    \eta_0=\min\left\{q_0,\lambda_{\Gamma}\rho_0\Gamma_{max}^{-1}\right\}/\max\left\{p_{max},\Gamma_{min}^{-1}\right\},
\end{equation}
where $\rho_0\in(0,1]$. Define a compact set $D$ as
\begin{equation}\label{e:D}
    D=\left\{\left(e,\tilde{\theta}\right)\hspace{.05cm}\middle|\hspace{.05cm}\eta\left[p_{min}\lVert e\rVert^2+\Gamma_{max}^{-1}\lVert\tilde{\theta}\rVert^2\right]\leq\upsilon\right\},
\end{equation}
and
\begin{equation}\label{e:D_max}
    D_{max}=\left\{\left(e,\tilde{\theta}\right)\hspace{.05cm}\middle|\hspace{.05cm}\eta_0\left[p_{min}\lVert e\rVert^2+\Gamma_{max}^{-1}\lVert\tilde{\theta}\rVert^2\right]\leq\upsilon_{max}\right\}.
\end{equation}
We now state the main theorem.
\begin{theorem}\label{theorem:Convergence}
    Under Assumptions \ref{a:e_e_y_V_V_dot} and \ref{a:theta_star}, the update laws in \eqref{e:theta_dot_AR}, \eqref{e:Gamma_Update}, and \eqref{e:Omega_Update} for the error model in \eqref{e:error_model_general} guarantee for any $\phi(t)$,
    \begin{enumerate}\renewcommand{\theenumi}{\Alph{enumi}}
        \item boundedness of the trajectories of $e(t)$ and $\tilde{\theta}(t)$, $\forall t\geq t_0$.
    \end{enumerate}
    If in addition $\phi$ is finitely exciting as in Assumption \ref{a:FE}, then
    \begin{enumerate}\setcounter{enumi}{1}\renewcommand{\theenumi}{\Alph{enumi}}
        \item the trajectories of $e(t)$, $\tilde{\theta}(t)$ tend exponentially fast towards a compact set $D\subset D_{max}$, $\forall t\in[t_3,~t_4]$.
    \end{enumerate}
    If in addition $\phi$ is persistently exciting $\forall t\geq t_1\geq t_0$ as in Definition \ref{d:PE} with level $\alpha\geq\alpha_0'$ and interval $T$, then
    \begin{enumerate}\setcounter{enumi}{2}\renewcommand{\theenumi}{\Alph{enumi}}
        \item exponential convergence of the trajectories follows, of $e(t)$, $\tilde{\theta}(t)$ towards a compact set $D\subset D_{max}$, $\forall t\geq t_3'$,
    \end{enumerate}
    where time instances and excitation levels are as in Lemmas \ref{l:Omega}, \ref{l:Gamma}.
\end{theorem}
\begin{proof}[Proof]
    Let $q_0=\min eig(Q)$, $p_{min}=\min eig(P)$, $p_{max}=\max eig(P)$. Consider a candidate Lyapunov function of the form
    \begin{equation}\label{e:V_MRAC_AR2}
        V(t)=e^T(t)Pe(t)+ Tr\left[\tilde{\theta}^T(t)\Gamma^{-1}(t)\tilde{\theta}(t)\right].
    \end{equation}
    It follows that
    \begin{align*}
        &\dot{V}(t)\hspace{-.05cm}\leq\hspace{-.05cm}\underbrace{-e^T(t)Qe(t)-2Tr\hspace{-.05cm}\left[\tilde{\theta}^T(t)Y(t)\right]\hspace{-.05cm}+ 2Tr\hspace{-.05cm}\left[\tilde{\theta}^T(t)\Gamma^{-1}(t)\dot{\theta}(t)\right]}_{\text{due to Assumption \ref{a:e_e_y_V_V_dot}}}\\
        &+\underbrace{Tr\left[\tilde{\theta}^T(t)\left\{\frac{d}{dt}\left(\Gamma^{-1}(t)\right)\right\}\tilde{\theta}(t)\right]}_{\text{due to time-varying }\Gamma(t)}- \underbrace{2Tr\left[\tilde{\theta}^T(t)\Gamma^{-1}(t)\dot{\theta}^*(t)\right]}_{\text{due to time-varying }\theta^*(t)}.
    \end{align*}
    Using \eqref{e:theta_dot_AR}, \eqref{e:Gamma_Update}, and Lemma \ref{l:Gamma_dot}, $\dot{V}(t)$ may be simplified as
    \begin{align*}
        \dot{V}(t)&\leq-e^T(t)Qe(t)-2Tr\left[\tilde{\theta}^T(t)\Gamma^{-1}(t)\dot{\theta}^*(t)\right]\\
        &+2Tr\left[\tilde{\theta}^T(t)\Gamma^{-1}(t)\left(\text{Proj}_{\Gamma(t)}\left(\theta(t),Y(t),F\right)-\Gamma(t)Y(t)\right)\right]\\
        &-\lambda_{\Gamma}Tr\left[\tilde{\theta}^T(t)\Gamma^{-1}(t)\text{Proj}\left(\Gamma(t),\mathcal{Y}(t),\mathcal{F}\right)\Gamma^{-1}(t)\tilde{\theta}(t)\right].
    \end{align*}
    Using Lemma \ref{l:Trace_theta_tilde_Proj_Y} and $\rho(t)$ in \eqref{e:rho}, we obtain that
    \begin{align}\label{e:V_dot_After_Alg}
        \begin{split}
        \dot{V}(t)&\leq-e^T(t)Qe(t)-2Tr\left[\tilde{\theta}^T(t)\Gamma^{-1}(t)\dot{\theta}^*(t)\right]\\
        &\quad-\lambda_{\Gamma}\rho(t)Tr\left[\tilde{\theta}^T(t)\left\{\Gamma^{-1}(t)-\kappa\Omega(t)\right\}\tilde{\theta}(t)\right].
        \end{split}
    \end{align}
    Using \eqref{e:upsilon}, \eqref{e:eta_t}, Corollary \ref{cor:tilde_theta_max}, and Assumption \ref{a:theta_star}, the inequality becomes
    \begin{equation}\label{e:V_dot_cases}
        \dot{V}(t)\hspace{-.06cm}\leq\hspace{-.06cm}\left\{\begin{array}{ll}
            \hspace{-.22cm}-q_0\lVert e(t)\rVert^2\hspace{-.02cm}+\hspace{-.02cm}2\Gamma_{min}^{-1}\tilde{\theta}_{max}\theta^*_{d,max}, & \hspace{-.25cm}\text{if }\rho(t)=0\\
            \hspace{-.22cm}-\eta(t)V(t)\hspace{-.02cm}+\hspace{-.02cm}\upsilon(t), & \hspace{-.25cm}\text{if }\rho(t)\in(0,1]
            \end{array}\right.
    \end{equation}
    From the first case of \eqref{e:V_dot_cases} it can be seen that $\dot{V}(t)\leq0$ for $\lVert e(t)\rVert\geq\sqrt{2\Gamma_{min}^{-1}\tilde{\theta}_{max}\theta^*_{d,max}/q_0}$. From Lemmas \ref{l:Omega}, \ref{l:Gamma}, \ref{l:theta_t}, and Corollary \ref{cor:tilde_theta_max}, each of $\Omega(t)$, $\Gamma(t)$, $\Gamma^{-1}(t)$, $\theta(t)$, $\theta^*(t)$ and $\tilde{\theta}(t)$ are bounded. Thus the trajectories of the closed loop system remain bounded. This proves Theorem \ref{theorem:Convergence}-A).
    
    From \eqref{e:V_MRAC_AR2} and \eqref{e:V_dot_cases}, it can be noted that $\dot{V}<0$ in $D^c$, where the compact set $D$ is defined in \eqref{e:D}. Applying the Comparison Lemma (see \cite{Khalil_2002}, Lemma 3.4) for the second case of \eqref{e:V_dot_cases}, we obtain that
    \begin{equation}\label{e:V_leq_Phi}
        V(t)\leq \Phi(t,t_3)V(t_3)+\int_{t_3}^{t}\Phi(t,\tau)\upsilon(\tau)d\tau,\quad\forall t\in[t_3,~t_4],
    \end{equation}
    with transition function $\Phi(t,\tau)=\text{exp}\left[-\int_{\tau}^t\eta(\tau)d\tau\right]$. It can be noted that from Lemma \ref{l:Gamma} it was shown that $\rho(t)\geq\rho_0$, $\forall t\in[t_3,~t_4]$, and thus $\eta(t)\geq\eta_0$, $\forall t\in[t_3,~t_4]$, which follows from \eqref{e:eta_t}, \eqref{e:eta_0}. Thus \eqref{e:V_leq_Phi} is simplified using \eqref{e:upsilon_max} as
    \begin{equation}\label{e:V_leq_eta_0}
        V(t)\leq \text{exp}\left[-\eta_0(t-t_3)\right]\left(\hspace{-.05cm}V(t_3)-\frac{\upsilon_{max}}{\eta_0}\hspace{-.05cm}\right)+\frac{\upsilon_{max}}{\eta_0},~\forall t\in[t_3,~t_4].
    \end{equation}
    Furthermore given that $\eta(t)\geq\eta_0$ $\forall t\in[t_3,~t_4]$ and $0\leq\upsilon(t)\leq\upsilon_{max}$, it can be noted that $D\subset D_{max}$, $\forall t\in[t_3,~t_4]$. Therefore it can be seen that over the interval of time $t\in[t_3,~t_4]$, the state error $e(t)$ and parameter error $\tilde{\theta}(t)$ tend exponentially fast towards the bounded set $D\subset D_{max}$. This proves Theorem \ref{theorem:Convergence}-B).
    
    If $\phi$ is persistently exciting with level $\alpha\geq\alpha_0'$, and interval $T$, then it follows from Lemma \ref{l:Gamma}-7) that \eqref{e:V_leq_Phi} and \eqref{e:V_leq_eta_0} hold for all $t\geq t_3'$, which proves Theorem \ref{theorem:Convergence}-C).
\end{proof}
\begin{remark}
    $\eta(t)$ denotes the convergence rate of $V(t)$. This in turn follows if $\rho(t)>0$, i.e. if $\Gamma(t)$ is bounded away from $\Gamma_{max}$. The latter follows from Lemma \ref{l:Omega}-3) and \ref{l:Omega}-4) if $\phi(t)$ is either finitely exciting or persistently exciting, with an exponentially fast trajectory of $V(t)$ towards a compact set occurring over a finite interval or for all $t\geq t_3$, respectively. This convergence rate however is upper bounded by $q_0/p_{max}$.
\end{remark}
\begin{remark}
    Theorem \ref{theorem:Convergence}-C) guarantees convergence of $V(t)$ to a compact set $D$, while Theorem \ref{theorem:Convergence}-B) guarantees that $V(t)$ approaches $D$. This set scales with the signal $\upsilon(t)$ in \eqref{e:upsilon}, which contains contributions both from $\Omega(t)$ and $\dot{\theta}^*(t)$. For static parameters ($\dot{\theta}^*(t)\equiv0$) and low excitation (i.e., $\rho(t)\neq0$ and $\lVert\Omega(t)\rVert<(1/(\kappa\Gamma_{max}))$), from \eqref{e:V_dot_After_Alg} it can be shown that the trajectories of $(e,\tilde{\theta})$, tend towards the origin, i.e. the set $(e,\tilde{\theta})=(0,0)$.
\end{remark}
\begin{remark}
    Since we did not introduce any filtering of the dynamics, the bound $\theta^*_{d,max}$ on the uncertain parameters is explicit in the compact set $D_{max}$. It can be seen that $D_{max}$ directly scales with $\theta^*_{d,max}$ from \eqref{e:upsilon_max}. Such an explicit bound cannot be derived using existing approaches in the literature which filter dynamics.
\end{remark}
\begin{remark}
    The dependence of $\upsilon(t)$ on $\dot{\theta}^*(t)$ is reasonable. As the time-variations in the uncertain parameters grow, it should be expected that the residue will increase as well. The dependence of $\upsilon(t)$ on the filtered regressor $\Omega(t)$ is introduced due to the structure of our algorithm in \eqref{e:theta_dot_AR}, \eqref{e:Gamma_Update}, and \eqref{e:Omega_Update}. As a result, even with persistent excitation, we can only conclude convergence of $V(t)$ to a compact set as opposed to convergence to the origin. This compact set will be present even in the absence of time-variations in $\theta^*$. This disadvantage, however, is offset by the property of exponential convergence to the compact set, which is virtue of the fact that we have a time-varying $\Gamma(t)$.
    
    A closer examination of the convergence properties of the proposed algorithm is worth carrying out for the case of constant parameters. It is clear from \eqref{e:V_dot_After_Alg} that the negative contributions to $\dot{V}(t)$ come from the first term, while any positive contribution comes if $\kappa\Omega(t)>\Gamma^{-1}(t)$. That is, if there is a large enough excitation, then the third term can be positive. This in turn is conservatively reflected in the magnitude of $\upsilon(t)$. It should however be noted that a large $\Omega(t)$ with persistent excitation, leads to a large $e(t)$, which implies that as the third term in \eqref{e:V_dot_After_Alg} becomes positive, it leads to a first term that is proportionately large and negative as well, thereby resulting in a net contribution that is negative. An analytical demonstration of this effect, however, is difficult to obtain. For this reason, the nature of our main result is convergence to a bounded set rather than to zero, in the presence of persistent excitation. Finally, we note that in our simulation studies, $\dot{V}(t)$ remained negative for almost the entire period of interest in the case of constant parameters, resulting in a steady convergence of the parameter estimation error to zero.
\end{remark}
\begin{remark}
    For a scalar adaptive system (for example in Figure \ref{f:robust_comparison}), initial conditions can be easily found such that the compact set for the well-known $\sigma$-mod \cite{Ioannou_1984}, $e$-mod \cite{Narendra_1987a} is larger than the compact set for the proposed method. This can be quantified but is beyond the scope of the present paper.
\end{remark}
\begin{remark}
    The bounds for $\Omega(t)$ and $\Gamma(t)$ derived in Lemmas \ref{l:Omega} and \ref{l:Gamma} were crucial in establishing the lower bound for $\eta(t)$. That this occurs over an interval of time results in exponential and not just asymptotic properties. This has obvious implications of robustness (c.f. \cite[\S 9.1]{Khalil_2002}).
\end{remark}

\section{Numerical Simulations}
\label{s:Comparison}

In this section we present two numerical simulations. In the first, we present results for F-16 longitudinal dynamics from \cite{Stevens2003}, trimmed and linearized at a straight and level flying condition with a velocity of $500$ ft/s and an altitude of $15,000$ ft. We further include integral tracking of the command signal thus resulting in the extended model
\begin{equation*}
\scalebox{0.78}{%
$\underbrace{
\begin{bmatrix}
\dot{x}_1(t)\\
\dot{x}_2(t)\\
\dot{x}_3(t)
\end{bmatrix}
}_{\dot{x}(t)}
=
\underbrace{
\begin{bmatrix}
-0.6398&~~0.9378&0\\
-1.5679&-0.8791&0\\
0&1&0\\
\end{bmatrix}
}_{A}
\underbrace{
\begin{bmatrix}
x_1(t)\\
x_2(t)\\
x_3(t)
\end{bmatrix}
}_{x(t)}
+
\underbrace{
\begin{bmatrix}
-0.0777\\
-6.5121\\
0
\end{bmatrix}
}_{B}
u(t)
+
\underbrace{
\begin{bmatrix}
0\\
0\\
-1
\end{bmatrix}
}_{B_z}
z_{cmd}(t)$}
\end{equation*}
where $z_{cmd}$ is the pitch rate command (dps), $u$ is an elevator deflection (deg), and the state variables $x_1,x_2,x_3$ are the angle of attack (deg), pitch rate (dps), and integrated pitch rate tracking error (deg), respectively. A reference model, representing the desired dynamics, is designed as $\dot{x}_m(t)=A_mx_m(t)+B_zz_{cmd}(t)$, where $A_m=A-BK^T$ is Hurwitz, with $K=[0.1965,~-0.3835,~-1]^T$. The tracking error dynamics may be expressed in the form of \eqref{e:error_model_general}, with error $e(t)=x_m(t)-x(t)$, $e_y(t)=e(t)$, $\phi(t)=x(t)$, and the control input selected as $u(t)=-K^Tx(t)-\theta^T(t)\phi(t)$. With the parameter update argument selected as $Y(t)=-\phi(t)e^T(t)PB$ (as in Table \ref{t:Common_Structure}), Assumption \ref{a:e_e_y_V_V_dot} may be verified.

The adaptive parameter estimate $\theta$ is initialized at zero, to estimate the nominal unknown parameter $\theta^*=[0.1965,~-0.03835,~0]^T$, which represents uncertainty as a function of angle of attack and pitch rate. For the algorithm in \eqref{e:theta_dot_AR}, \eqref{e:Gamma_Update}, \eqref{e:Omega_Update}, we set $\lambda_{\Gamma}=0.5$, $\kappa=0.5$, $\lambda_{\Omega}=10$, and find a matrix $P$ which solves $A_m^TP+PA_m=-I$. The time-varying learning rate is initialized as $\Gamma(t_0)=\Gamma_0=10I$. For the projection algorithms in \eqref{e:theta_dot_AR} and \eqref{e:Gamma_Update}, we use the convex, coercive continuous function in \eqref{e:Projection_AC} where the 2-norm is used for the $\theta$ update and the Frobenius norm is used for the $\Gamma$ update.

We first present results for the case of a constant unknown parameter in order to demonstrate exponential convergence properties \emph{towards the origin} with finite excitations, and then proceed to the case when the unknown parameters are time-varying. The results in Figure \ref{sf:State_Control} demonstrate command tracking of pitch rate step responses using the standard MRAC update with a constant $\Gamma$ (as in \eqref{e:theta_dot_nom}) and the time-varying learning rate update (TR-MRAC) proposed in this paper in \eqref{e:theta_dot_AR}-\eqref{e:Omega_Update}, for the case of a constant $\theta^*$. It can be noted that while step responses contain some frequency content, the regressor $\phi(t)$ does not satisfy the PE condition in Definition \ref{d:PE} with a large $\alpha$. This is reflected in Figure \ref{sf:Lyapunov_Gamma}, as the parameter errors remain constant for the static learning rate method. However, as time proceeds, with the help of \eqref{e:Omega_Update}, the excitation content is retained and increased, leading to a change in $\Gamma(t)$ for the time-varying learning rate update, which is immediately followed by a fast decrease in $V(t)$ (note the log-scale in Figure \ref{sf:Lyapunov_Gamma}). It is this interplay between excitation, change in $\Gamma(t)$, and a large decrease in $V$ that ensures the exponential convergence of errors to a compact set.
\begin{figure}[!t]
    \begin{subfigure}[b]{\columnwidth}
        \centerline{
	    \includegraphics[trim={0.4cm 2.3cm 1.9cm 8.4cm},clip,width=0.45\columnwidth]{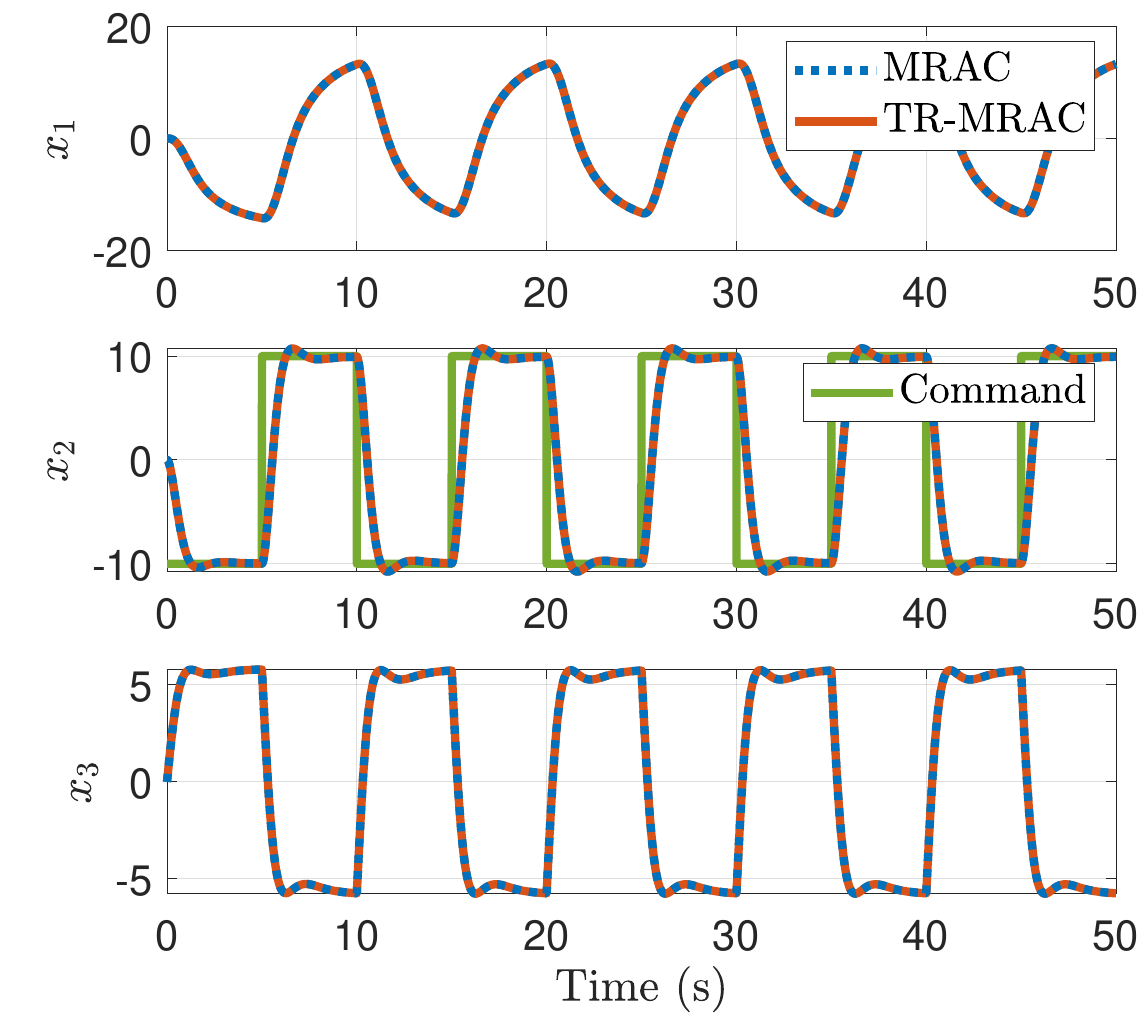}
	    \includegraphics[trim={0.4cm 2.3cm 1.9cm 8.4cm},clip,width=0.45\columnwidth]{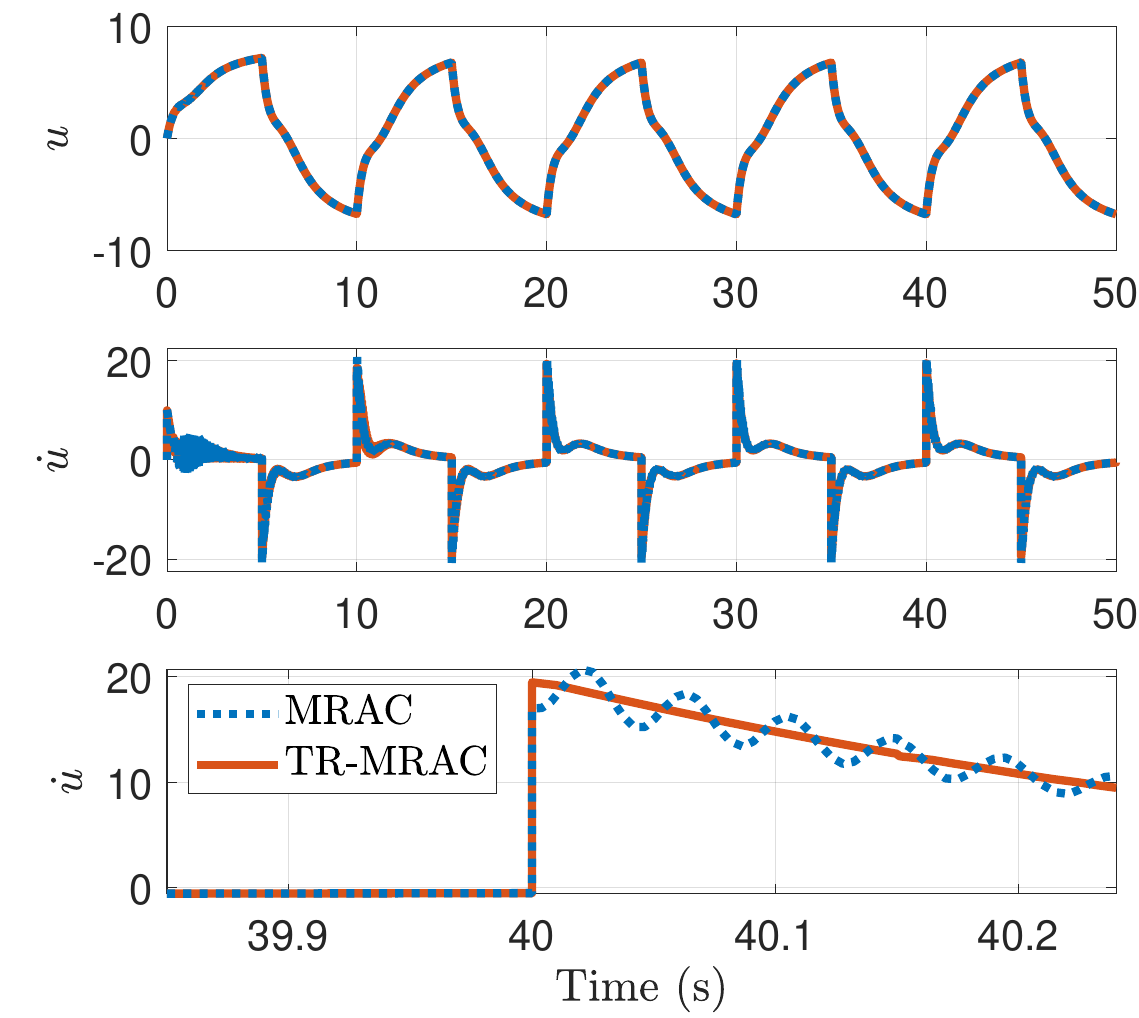}
	    }
        \caption{State and control time histories.}
        \label{sf:State_Control}
    \end{subfigure}
    
    \begin{subfigure}[b]{\columnwidth}
        \centerline{
	    \includegraphics[trim={0.4cm 2.3cm 1.9cm 8.4cm},clip,width=0.45\columnwidth]{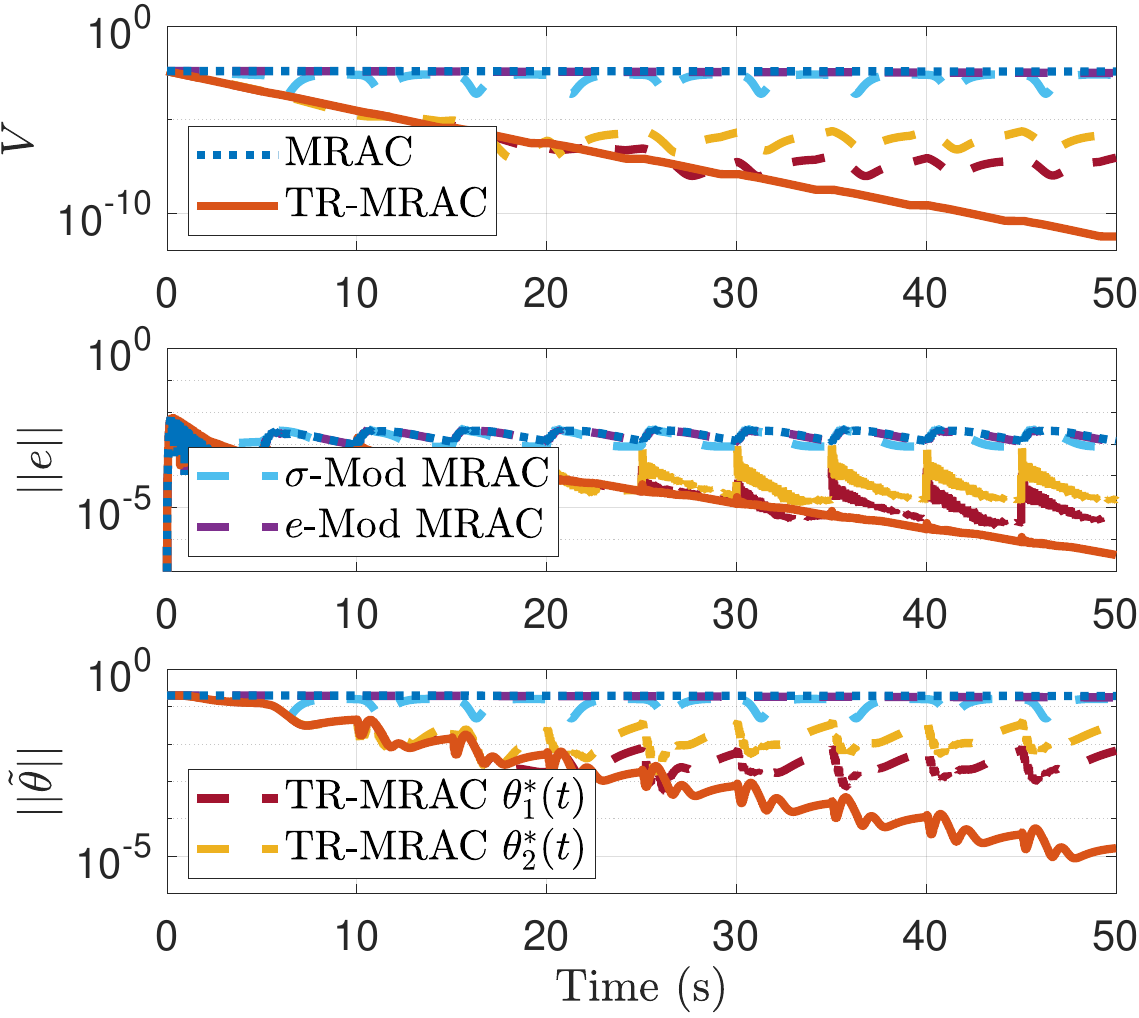}
	    \includegraphics[trim={0.4cm 2.3cm 1.9cm 8.4cm},clip,width=0.45\columnwidth]{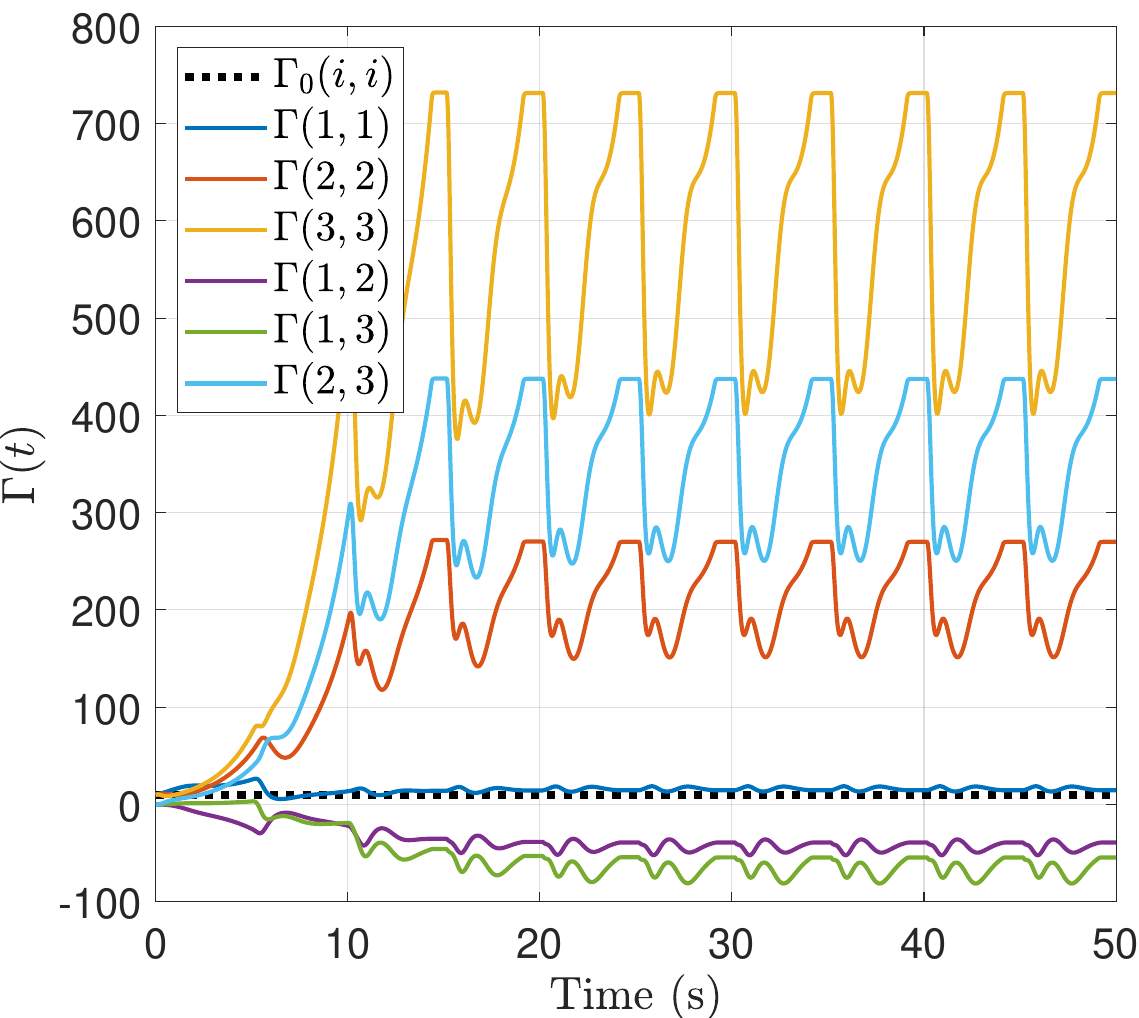}
	    }
        \caption{Lyapunov function, tracking error norm, parameter error norm and time-varying learning rate time histories.}
        \label{sf:Lyapunov_Gamma}
    \end{subfigure}
    \caption{Time histories of numerical simulation.}
    \label{f:simulation_plots}
\end{figure}
\begin{figure}[!t]
        \centering
        \begin{minipage}{0.25\columnwidth}
        \centerline{
        \includegraphics[trim={3.4cm 3.8cm 4.4cm 4.75cm},clip,width=\columnwidth]{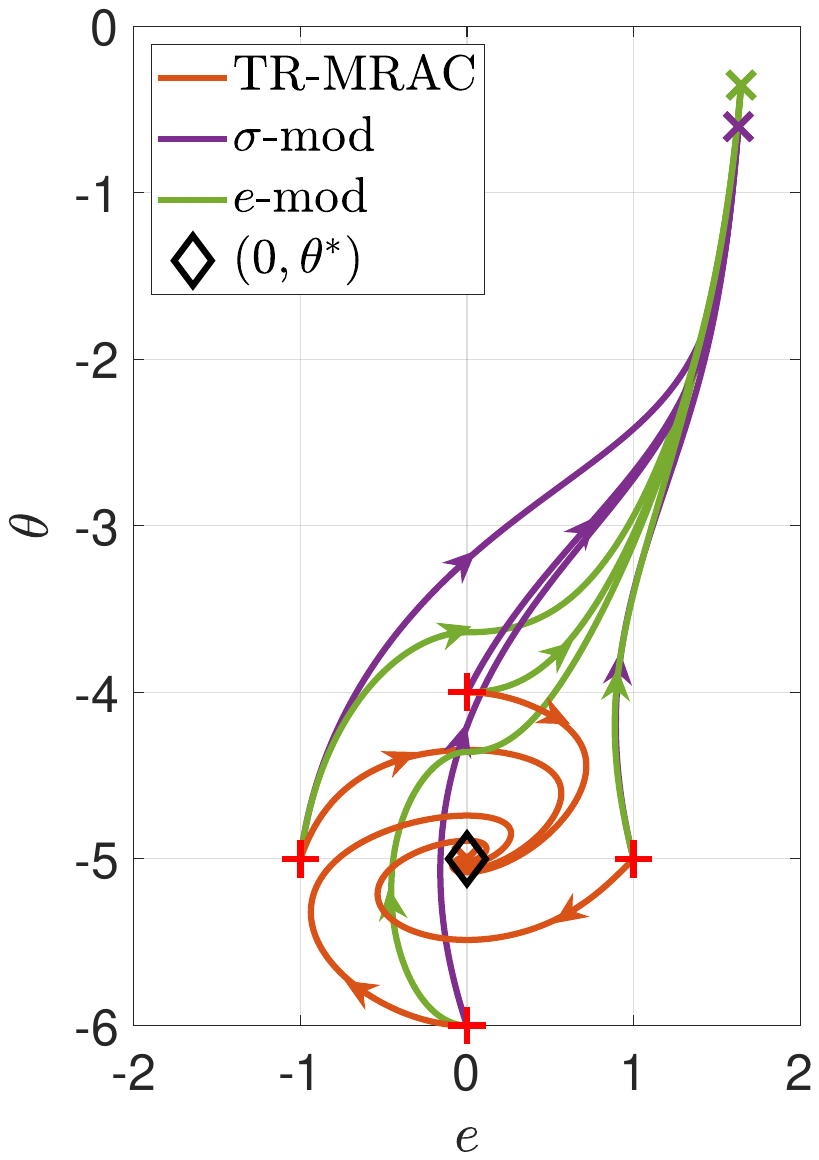}
        }
        \end{minipage}%
        \begin{minipage}{0.75\columnwidth}
        \vspace{-.4cm}
        \begin{align*}
            \text{error model}&:\dot{e}(t)=-e(t)+\tilde{\theta}(t)\phi(t)\\
            \text{TR-MRAC}&:\dot{\theta}(t)=-\gamma(t)e(t)\phi(t)\\
            \text{$\sigma$-mod}&:\dot{\theta}(t)=-e(t)\phi(t)-\theta(t)\\
            \text{$e$-mod}&:\dot{\theta}(t)=-e(t)\phi(t)-|e(t)|\theta(t)\\
            \text{regressor}&:\phi(t)=2-e(t)
        \end{align*}
        \end{minipage}
    \caption{Trajectories of a first order PE adaptive system. Initial conditions marked by red ``$+$'', equilibrium points marked by ``$\times$''.}
    \label{f:robust_comparison}
\end{figure}

Figure \ref{sf:Lyapunov_Gamma} (left) also includes results for the $\sigma$-mod \cite{Ioannou_1984}, $e$-mod \cite{Narendra_1987a} adaptive update modifications. It can be noted that these modifications do not greatly improve the error convergence rate.

We further include, in Figure \ref{sf:Lyapunov_Gamma} (left), results obtained with two sets of time-varying parameters $\theta^*_1(t)=(1+1.1t/50)\theta^*$ (in dark red) and $\theta^*_2(t)=(1+1.5t/50)\theta^*$ (in yellow). As all plots show, the errors decrease despite the time-variations. The larger time-variation results in a larger error as shown in Figure \ref{sf:Lyapunov_Gamma} (left).

As a benchmark, we consider a first-order plant example with a single unknown parameter (see \cite{Ioannou_1984} and \cite{Narendra_1987a} for details of the setup), the corresponding phase plots in the $(e,\theta)$ space of which are shown in Figure \ref{f:robust_comparison}, for the $\sigma$-mod \cite{Ioannou_1984}, $e$-mod \cite{Narendra_1987a}, and the proposed TR-MRAC which is designed with $0<\gamma_{min}\leq\gamma(t)\leq1$. Figure \ref{f:robust_comparison} illustrates that the $\sigma$-mod and $e$-mod create spurious equilibrium points to which all trajectories shown with these two algorithms converge. In contrast, TR-MRAC has a sole equilibrium point at $(0,-5)$ to which all trajectories converge.

\section{Concluding Remarks}
\label{s:Conclusion}

In this paper we presented a new parameter estimation algorithm for the adaptive control of a class of time-varying plants. The main feature of this algorithm is a matrix of time-varying learning rates, which enables exponentially fast trajectories of parameter estimation errors towards a compact set whenever excitation conditions are satisfied. It is shown that even in the presence of time-varying parameters, this algorithm guarantees global boundedness of the state and parameter errors of the system. The learning rate matrix is ensured to be bounded through the use of a projection operator. Since no filtering is employed and the original dynamic structure of the system is preserved, the bounds derived are tractable and are clearly related to the bounds on the time-variations of the unknown parameters as well as the excitation properties. Numerical simulations were provided to complement the theoretical analysis. Future work will focus on connecting these time-varying learning rates to accelerated learning in machine learning problems. 

\appendices

\section*{Proofs of Lemmas}
\label{s:Proofs_Lemmas}

\begin{proof}[Proof of Lemma \ref{l:Xi}]
    Let $x_1,x_2\in\Xi_{\delta}$ and thus $f(x_1)\leq\delta$ and $f(x_2)\leq\delta$. From the convexity of $f$, for any $0\leq\lambda\leq 1$:\\
    $f(\lambda x_1+(1-\lambda)x_2)\leq\lambda f(x_1)+(1-\lambda) f(x_2)\leq\lambda\delta+(1-\lambda)\delta=\delta$.
    Therefore for all $x=\lambda x_1+(1-\lambda)x_2$: $f(x)\leq\delta$, and thus $x\in\Xi_{\delta}$. Therefore $\Xi_{\delta}$ is a convex set.
\end{proof}
\begin{proof}[Proof of Lemma \ref{l:bounded_sublevel}]
    Suppose there exists a constant $\delta>0$ such that the subset $\Xi_{\delta}=\{x\in\mathbb{R}^N\mid f(x)\leq\delta\}$ is nonempty and unbounded. Thus there exists a sequence $\{x_k\mid k\in\mathbb{N}\}\in\Xi_{\delta}$ such that $\lVert x_k\rVert\rightarrow\infty$. From Definition \ref{d:Coercive}, given that $f$ is coercive then $\lim_{k\rightarrow\infty}f(x_k)=\infty$. This contradicts $f(x)\leq\delta,~ \forall x\in \Xi_{\delta}$.
\end{proof}
\begin{proof}[Proof of Lemma \ref{l:theta_star_theta_b}]
    $f(\theta)$ is convex, thus for any $0<\lambda\leq1$: $f(\lambda\theta^*+(1-\lambda)\theta)=f(\theta+\lambda(\theta^*-\theta))\leq f(\theta)+\lambda(f(\theta^*)-f(\theta))$. Thus $(\theta-\theta^*)^T\nabla f(\theta)=\lim_{\lambda\rightarrow0}\left(f(\theta)-f(\theta+\lambda(\theta^*-\theta))\right)/\lambda\geq f(\theta)-f(\theta^*)\geq \delta-\delta=0$.
\end{proof}
\begin{proof}[Proof of Lemma \ref{l:Trace_theta_tilde_Proj_Y}]
    For each $j\in1,\ldots,m$, if $f_j(\theta_j)>0\wedge y_j^T\Gamma\nabla f_j(\theta_j)>0$, then using \eqref{e:Projection_Vector_Gamma} and Lemma \ref{l:theta_star_theta_b},
    \begin{align*}
        &Tr\left[(\theta-\theta^*)^T\Gamma^{-1}\left(\text{Proj}_{\Gamma}(\theta,Y,F)-\Gamma Y\right)\right]\\
        &\quad=\sum_{j=1}^m(\theta_j-\theta_j^*)^T\Gamma^{-1}\left(\text{Proj}_{\Gamma}(\theta_j,y_j,f_j)-\Gamma y_j\right)\\
        &\quad=-\sum_{j=1}^m\frac{\underbrace{(\theta_j-\theta_j^*)^T\nabla f_j(\theta_j)}_{\geq0}\underbrace{(\nabla f_j(\theta_j))^T\Gamma y_j}_{>0}}{\underbrace{(\nabla f_j(\theta_j))^T\Gamma\nabla f_j(\theta_j)}_{>0}}\underbrace{f_j(\theta_j)}_{>0}\leq0
    \end{align*}
    otherwise $(\theta_j-\theta_j^*)^T\Gamma^{-1}\left(\text{Proj}_{\Gamma}(\theta_j,y_j,f_j)-\Gamma y_j\right)=0$.
\end{proof}
\begin{proof}[Proof of Lemma \ref{l:Gamma_dot}]
    It can be noticed that $0=\frac{d}{dt}\left(I\right)=\frac{d}{dt}\left(\Gamma(t)\Gamma^{-1}(t)\right)=\dot{\Gamma}(t)\Gamma^{-1}(t)+\Gamma(t)\left[\frac{d}{dt}\left(\Gamma^{-1}(t)\right)\right]$. Lemma \ref{l:Gamma_dot} follows by solving for $\frac{d}{dt}\left(\Gamma^{-1}(t)\right)$.
\end{proof}
\begin{proof}[Proof of Lemma \ref{l:Omega}]
    Let $v\in\mathbb{R}^N$. Given the initial condition for \eqref{e:Omega_Update}, it can be noted that $0\leq v^T\Omega(t_0)v\leq \lVert v\rVert^2$. Furthermore
    \begin{equation}\label{e:phi_phiT_normalized}
        0\leq\frac{\lvert v^T\phi(t)\rvert^2}{1+\lVert\phi(t)\rVert^2}\leq \lVert v\rVert^2,\quad \forall v,t\geq t_0,
    \end{equation}
    as multiplying through by $1+\lVert\phi(t)\rVert^2\geq1$, the lower and upper bounds may be shown as $0\leq \lvert v^T\phi(t)\rvert^2$, and $\lvert v^T\phi(t)\rvert^2\leq \lVert v\rVert^2\lVert \phi(t)\rVert^2 \leq\lVert v\rVert^2+\lVert v\rVert^2\lVert \phi(t)\rVert^2$, $\forall v,t\geq t_0$.
    
    1) From the integral update in \eqref{e:Omega_Update} we obtain
    \begin{align}\label{e:vTOmegav}
        \begin{split}
            v^T\Omega(t)v&=\text{e}^{-\lambda_{\Omega}(t-t_0)}v^T\Omega(t_0)v\\
            &\quad+\int_{t_0}^t\text{e}^{-\lambda_{\Omega}(t-\tau)}\lambda_{\Omega}\frac{\lvert v^T\phi(\tau)\rvert^2}{1+\lVert\phi(\tau)\rVert^2}d\tau.
        \end{split}
    \end{align}
    Given that $\lambda_{\Omega}>0$, using \eqref{e:phi_phiT_normalized}, all terms in \eqref{e:vTOmegav} are non-negative, therefore $v^T\Omega(t)v\geq0$, $\forall v,t\geq t_0$ and thus $\Omega(t)\geq0$, $\forall t\geq t_0$.
    
    2) From the integral update in equation \eqref{e:Omega_Update} we obtain $v^T(I-\Omega(t))v=v^T\left(I-\text{exp}(-\lambda_{\Omega}(t-t_0))\Omega(t_0)\right)v-\int_{t_0}^t\text{exp}(-\lambda_{\Omega}(t-\tau))\lambda_{\Omega}\lvert v^T\phi(\tau)\rvert^2/(1+\lVert\phi(\tau)\rVert^2)d\tau$, which may be bounded using $\lambda_{\Omega}>0$ and \eqref{e:phi_phiT_normalized} to result in $v^T(I-\Omega(t))v\geq 0$, $\forall v,t\geq t_0$ and thus $\Omega(t)\leq I$, $\forall t\geq t_0$.
    
    3) The left hand side of \eqref{e:vTOmegav} may be lower bounded as $v^T\Omega(t_2)v\geq\text{exp}(-\lambda_{\Omega}(t_2-t_1))\lambda_{\Omega}v^T\int_{t_1}^{t_2}(\phi(\tau)\phi^T(\tau))/(1+\lVert\phi(\tau)\rVert^2)d\tau v$. Thus using the finite excitation condition in Assumption \ref{a:FE} (see Definition \ref{d:FE}), $v^T\Omega(t_2)v\geq v^T(\lambda_{\Omega}\alpha/d)\text{exp}(-\lambda_{\Omega}(t_2-t_1))Iv\geq(k_{\Omega}/(\rho_{\Omega}\kappa\Gamma_{max}))v^TIv$. Furthermore from update \eqref{e:Omega_Update}: $v^T\Omega(t)v\geq \text{exp}(-\lambda_{\Omega}(t-t_2))v^T\Omega(t_2)v$, $\forall v,t\geq t_2$. Therefore $v^T\Omega(t)v\geq\text{exp}(-\lambda_{\Omega}(t-t_2))(k_{\Omega}/(\rho_{\Omega}\kappa\Gamma_{max}))v^TIv$, $\forall v,t\geq t_2$. Therefore $\Omega(t)\geq(k_{\Omega}/(\kappa\Gamma_{max}))I>(1/(\kappa\Gamma_{max}))$, $\forall t\in[t_2,~t_3]$.
    
    4) Immediate from extension of the proof of case 3) $\forall t\geq t_2'$.
\end{proof}

\begin{proof}[Proof of Lemma \ref{l:Gamma}]
    1) The time derivative for $\mathcal{F}(\Gamma)$ in \eqref{e:Gamma_Update} is $\dot{\mathcal{F}}(\Gamma)=Tr[(\nabla \mathcal{F}(\Gamma))^T\dot{\Gamma}]=\lambda_{\Gamma}Tr[(\text{Proj}(\Gamma,\mathcal{Y},\mathcal{F}))^T\nabla \mathcal{F}(\Gamma)]$. With the projection equation in \eqref{e:Projection_Matrix_PD},
    \begin{equation*}
        \scalebox{0.88}{%
        $\dot{\mathcal{F}}(\Gamma)=
        \left\{\begin{array}{ll}
        \hspace{-.15cm}\lambda_{\Gamma}Tr[\mathcal{Y}^T\nabla \mathcal{F}(\Gamma)](1-\mathcal{F}(\Gamma)), & \hspace{-.25cm}\mathcal{F}(\Gamma)>0\wedge Tr[\mathcal{Y}^T\nabla \mathcal{F}(\Gamma)]>0\\
        \hspace{-.15cm}\lambda_{\Gamma}Tr[\mathcal{Y}^T\nabla \mathcal{F}(\Gamma)], & \hspace{-.25cm}\text{otherwise}
        \end{array}\right.$}
    \end{equation*}
    Therefore within the limiting region $\mathcal{F}(\Gamma)>0$,
    \begin{equation}\label{e:Limiting_F_Gamma}
        \left\{\begin{array}{ll}
        \hspace{-.15cm}\dot{\mathcal{F}}(\Gamma)>0, & 0<\mathcal{F}(\Gamma)<1\wedge Tr[\mathcal{Y}^T\nabla \mathcal{F}(\Gamma)]>0\\
        \hspace{-.15cm}\dot{\mathcal{F}}(\Gamma)=0, & \mathcal{F}(\Gamma)=1\wedge Tr[\mathcal{Y}^T\nabla \mathcal{F}(\Gamma)]>0\\
        \hspace{-.15cm}\dot{\mathcal{F}}(\Gamma)\leq0, & Tr[\mathcal{Y}^T\nabla \mathcal{F}(\Gamma)]\leq0
        \end{array}\right.
    \end{equation}
    Thus given the initial condition for \eqref{e:Gamma_Update}, $\mathcal{F}(\Gamma(t_0))\leq1$ and therefore $\mathcal{F}(\Gamma(t))\leq 1$ for all $t\geq t_0$. Therefore from Lemma \ref{l:bounded_sublevel}, there exists a constant denoted $\Gamma_{max}$ such that $\lVert\Gamma\rVert\leq\Gamma_{max}$ for all $\Gamma\in\Upsilon_1$, i.e. all $\mathcal{F}(\Gamma)\leq1$, $\forall t\geq t_0$, which implies $\Gamma(t)\leq\Gamma_{max}I$, $\forall t\geq t_0$.
    
    2) From \eqref{e:Limiting_F_Gamma} and Remark \ref{r:rho}, $\rho(t)\in[0,1]$, $\forall t\geq t_0$.
    
    From \eqref{e:Projection_Matrix_PD}, \eqref{e:Gamma_Update}, Remark \ref{r:rho} and Lemma \ref{l:Gamma_dot}, the inverse of the time-varying learning rate may be expressed as
    \begin{equation}\label{e:Gamma_inv_t}
        \frac{d}{dt}\left(\Gamma^{-1}(t)\right)=-\lambda_{\Gamma}\rho(t)\Gamma^{-1}(t)+\lambda_{\Gamma}\rho(t)\kappa\Omega(t).
    \end{equation}
    Let $v\in\mathbb{R}^N$. From \eqref{e:Gamma_inv_t} we obtain
    \begin{align}\label{e:Gamma_Inverse}
        \begin{split}
            v^T\Gamma^{-1}(t)v&=\text{e}^{-\lambda_{\Gamma}\int_{t_0}^t\rho(\tau)d\tau}v^T\Gamma^{-1}(t_0)v\\
            &\quad+\int_{t_0}^t\text{e}^{-\lambda_{\Gamma}\int_{\tau}^t\rho(\nu)d\nu}\lambda_{\Gamma}\rho(\tau)\kappa v^T\Omega(\tau)vd\tau.
        \end{split}
    \end{align}
    
    3) Using $\Omega(t)\leq I$, $\forall t\geq t_0$ from Lemma \ref{l:Omega} and $\rho(t)\in[0,1]$, $\forall t\geq t_0$, with equation \eqref{e:Gamma_Inverse} we obtain $v^T\left(\Gamma_{min}^{-1}I-\Gamma^{-1}(t)\right)v\geq0$, $\forall v,t\geq t_0$ and thus $\Gamma(t)\geq\Gamma_{min}I>0$, $\forall t\geq t_0$.
    
    4) From the excitation lower bound: $\Omega(t)\geq(k_{\Omega}/(\kappa\Gamma_{max}))I>(1/(\kappa\Gamma_{max}))I$, $\forall t\in[t_2,~t_3]$ from Lemma \ref{l:Omega}, it can be noted with equation \eqref{e:Gamma_Update} that $\Gamma_{max}-\kappa\Gamma_{max}\Omega(t)\Gamma_{max}\leq-\Gamma_{max}(k_{\Omega}-1)<0$, $\forall t\in[t_2,~t_3]$, as $k_{\Omega}>1$. Thus even if $\Gamma(t)$ were to be at the limit $\Gamma_{max}$, the evolution of $\Gamma(t)$ is bounded away from $\Gamma_{max}$ (towards $\Gamma_{max}/k_{\Omega}$). Therefore there exists a $\rho_{t_3}\in(0,1)$, such that $\rho(t)\geq\rho_{t_3}$, $\forall t\in[t_2,~t_3]$, and a $\Gamma_{t_3}<\Gamma_{max}$, such that $\Gamma(t)\leq\Gamma_{t_3}$, $\forall t\in[t_2,~t_3]$. Furthermore from \eqref{e:Gamma_Inverse}: $v^T\Gamma^{-1}(t)v\geq\text{exp}(-\lambda_{\Gamma}\int_{t_3}^t\rho(\tau)d\tau)v^T\Gamma^{-1}(t_3)v$, $\forall v,t\geq t_3$. Thus $v^T\Gamma^{-1}(t)v\geq\text{exp}(-\lambda_{\Gamma}(t-t_3))\Gamma^{-1}_{t_3}v^TIv$, $\forall v,t\geq t_3$. Therefore $\Gamma^{-1}(t)\geq \Gamma_{FE}^{-1}I$, $\forall t\in[t_3,~t_4]$, thus $\Gamma(t)\leq \Gamma_{FE}I<\Gamma_{max}$, $\forall t\in[t_3,~t_4]$.
    
    5) Given that $\Gamma(t)\leq \Gamma_{FE}I<\Gamma_{max}$, $\forall t\in[t_3,~t_4]$, it can be seen that $\exists \rho_0\in(0,1]$ such that $\rho(t)\geq\rho_0>0$, $\forall t\in[t_3,~t_4]$.
    
    6) Immediate from extension of the proof of case 4) $\forall t\geq t_3'$.
    
    7) Immediate from extension of the proof of case 5) $\forall t\geq t_3'$.
\end{proof}
\begin{proof}[Proof of Lemma \ref{l:theta_t}]
    For each $j\in1,\ldots,m$, the time derivative for $f_j(\theta_j)$ in \eqref{e:theta_dot_AR} may be expressed as $\dot{f}_j(\theta_j)=(\nabla f_j(\theta_j))^T\dot{\theta}_j=\left(\text{Proj}_{\Gamma}\left(\theta_j,y_j,f_j\right)\right)^T\nabla f_j(\theta_j)$. With the projection equation in \eqref{e:Projection_Vector_Gamma},
    \begin{equation*}
    \scalebox{0.90}{%
        $\dot{f}_j(\theta_j)=
        \left\{\begin{array}{ll}
        \hspace{-.15cm}y_j^T\Gamma\nabla f_j(\theta_j)(1-f_j(\theta_j)), & f_j(\theta_j)>0\wedge y_j^T\Gamma\nabla f_j(\theta_j)>0\\
        \hspace{-.15cm}y_j^T\Gamma\nabla f_j(\theta_j), & \text{otherwise}
        \end{array}\right.$}
    \end{equation*}
    Therefore within the limiting region $f_j(\theta_j)>0$,
    \begin{equation*}
        \left\{\begin{array}{ll}
        \hspace{-.15cm}\dot{f}_j(\theta_j)>0, & 0<f_j(\theta_j)<1\wedge y_j^T\Gamma\nabla f_j(\theta_j)>0\\
        \hspace{-.15cm}\dot{f}_j(\theta_j)=0, & f_j(\theta_j)=1\wedge y_j^T\Gamma\nabla f_j(\theta_j)>0\\
        \hspace{-.15cm}\dot{f}_j(\theta_j)\leq0, & y_j^T\Gamma\nabla f_j(\theta_j)\leq0
        \end{array}\right.
    \end{equation*}
    Thus given the initial condition for \eqref{e:theta_dot_AR}, $f_j(\theta_j(t_0))\leq1$ and therefore $f_j(\theta_j(t))\leq 1$ for all $t\geq t_0$. Therefore from Lemma \ref{l:bounded_sublevel}, there exists constants denoted $\theta_{j,max}$ such that $\lVert\theta_j(t)\rVert\leq\theta_{j,max}$ for all $\theta_j(t)\in\Xi_{1,j}$, i.e. all $f_j(\theta_j)\leq1$, $\forall t\geq t_0$, as proven. Thus there exists a constant denoted $\theta_{max}$ such that $\lVert\theta(t)\rVert\leq\theta_{max}$, $\forall t\geq t_0$.
\end{proof}

\begin{table}[t]
\caption{Core variable nomenclature}
\centering
\begin{tabular}{cc|cc}
\toprule
Learning Rate & $\Gamma$ & Projection Scalar & $\rho$ \\
Information Matrix & $\Omega$ & $\Gamma$-Projection Operator & $\text{Proj}_{\Gamma}$ \\
Tracking Error & $e$ & User-Defined Gains & $\lambda_{\Gamma},\kappa,\lambda_{\Omega}$ \\
Parameter Estimation Error & $\tilde{\theta}$ & Excitation Level & $\alpha_0$ \\
Parameter Estimate & $\theta$ & Compact set & $D$ \\
True Unknown Parameter & $\theta^*$ & Compact Set Scalar & $\upsilon$ \\
Regressor & $\phi$ & Convergence Rate & $\eta$ \\
\bottomrule
\end{tabular}
\label{t:Nomenclature}
\end{table}

\bibliographystyle{IEEEtran}
\bibliography{IEEEabrv,References}

\end{document}